\documentclass{amsart}
\usepackage{amssymb}
\usepackage{enumitem}
\usepackage{hyperref}
\usepackage{amsrefs}
\usepackage{tikz-cd}

\newtheorem{theorem}{Theorem}[section]
\newtheorem{lemma}[theorem]{Lemma}
\newtheorem{corollary}[theorem]{Corollary}
\newtheorem{proposition}[theorem]{Proposition}

\theoremstyle{definition}
\newtheorem{definition}[theorem]{Definition}
\newtheorem{example}[theorem]{Example}

\theoremstyle{remark}

\numberwithin{equation}{section}

\newcommand{\oo}{\infty}

\DeclareMathOperator*{\sgn}{sgn}
\DeclareMathOperator{\Int}{Int}

\begin{document}
\title[Core decreasing functions]{Core decreasing functions}
\author{Alejandro Santacruz Hidalgo}
\address{Department of Mathematics, Western University, London, Canada}
\email{asantacr@uwo.ca}

\author{Gord Sinnamon}
\address{Department of Mathematics, Western University, London, Canada}
\email{sinnamon@uwo.ca}
\thanks{The second author was supported by the Natural Sciences and Engineering Research Council of Canada}


\begin{abstract}  Given a measure space and a totally ordered collection of measurable sets, called an ordered core, the notion of a core decreasing function is introduced and used to define the down space of a Banach function space. This is done using a variant of the K\"othe dual restricted to core decreasing functions. To study down spaces, the least core decreasing majorant construction and the level function construction, already known for functions on the real line, are extended to this general setting. These are used to give concrete descriptions of the duals of the down spaces and, in the case of universally rearrangement invariant (u.r.i.) spaces, of the down spaces themselves.

The down spaces of $L^1$ and $L^\oo$ are shown to form an exact Calder\'on couple with divisibility constant $1$; a complete description of the exact interpolation spaces for the couple is given in terms of level functions; and the down spaces of u.r.i. spaces are shown to be precisely those interpolation spaces that have the Fatou property. The dual couple is also an exact Calder\'on couple with divisibility constant $1$; a complete description of the exact interpolation spaces for the couple is given in terms of least core decreasing majorants; and the duals of down spaces of u.r.i. spaces are shown to be precisely those interpolation spaces that have the Fatou property. 
\end{abstract}
\maketitle

\section{Introduction}\label{intro} Monotone functions on $\mathbb R$ are very well behaved compared to general measurable functions. Consequently, a wide variety of techniques and applications are in place for working with them. Our object is to investigate certain functional analysis tools involving monotone functions, making them available for functions on general measure spaces in which a highly customizable notion of order is used to determine monotonicity. 

Let $1\le p\le\oo$ and $1/p+1/p'=1$. If $f$ is a Lebesgue measurable function on $[0,\oo)$, there exists a nonnegative, nonincreasing function $f^o$, called the \emph{level function} of $f$, such that \[
\int_0^\oo|f|g\le\int_0^\oo f^og
\]
holds for all nonnegative, nonincreasing $g$, and
\[
\|f^o\|_{L^p}=\sup\bigg\{\int_0^\oo|f|g:\|g\|_{L^{p'}}\le1,g\ge0,g\text{ nonincreasing}\bigg\}.
\]
Here $\|\cdot\|_{L^p}$ denotes the usual $L^p$ norm. This improves H\"older's inequality in the presence of monotonicity since we have $\int_0^\oo|f|g\le\|f^o\|_{L^p}\|g\|_{L^{p'}}$ whenever $g$ is nonnegative and nonincreasing. To take advantage of this improvement, we need to understand $f^o$: For a locally integrable $f$, the function $f^o$ is determined by requiring that the function $x\mapsto\int_0^xf^o$ is the least concave majorant of $x\mapsto\int_0^x|f|$. In particular, $f^o$ is independent of $p$. The construction of $f^o$ is due to Halperin, in \cite{H}, with an alternate proof given by Lorentz, in \cite{L}, and their construction applies to weighted Lebesgue measure as well as to the unweighted case outlined above.

The construction of $f^o$ has been extended to functions on $(\mathbb R,\lambda)$, for a general Borel measure $\lambda$, and to function spaces more general than the $L^p$ spaces. It has been applied to give formulas for the dual spaces of Lorentz and Orlicz-Lorentz spaces, to prove weighted Hardy and Fourier inequalities, to transfer monotonicity (from kernel to weight) in weighted norm inequalities for general positive integral operators, and to provide equivalent norms for traditional and abstract Ces\'aro spaces that facilitate interpolation of these spaces and of their duals. Another construction determines the dual spaces of spaces defined by the level function, and strong interpolation results have been established for both scales of spaces. See, for example, \cites{HKM,KR,Le,LM,MS1,MS2,RS,Sil,Slr,Stm,Sft,Snl,Sp}. Additional references may be found in \cite{FLM}.

These powerful tools are currently available only for functions defined on $(\mathbb R,\lambda)$, where the natural order on $\mathbb R$ determines the collection of nonnegative, nonincreasing functions. (Requiring $\lambda$ to be Borel ensures that such functions are measurable.) Recently, in \cite{Snf}, the notion of a measure space with an \emph{ordered core} was introduced to study abstract Hardy operators. Here we use an ordered core to define a collection of nonnegative, nonincreasing functions (called core decreasing functions) on a general measure space and investigate the level function construction, the function spaces it generates, their dual spaces, and the interpolation properties of both. The ordered core can be chosen to suit the investigation, and all these tools will be adapted to that choice.

We will end this introduction by defining, below, the principal objects of study and mentioning a few easy-to-state results of our investigation. Section \ref{background} sets out some necessary background and relevant known results. In Section \ref{Core Maps}, one of the main tools used to study ordered cores in \cite{Snf} is extended and adapted to our purposes. Much of the measure theory needed in the paper is contained in Section \ref{Core Functions}, where we show how to enrich an ordered core without changing its essential order properties. Examples to illustrate the flexibility of this notion of order are included in Section \ref{Examples}, along with pair of key examples that exhibit behaviour quite unlike the usual situation in which the order is carried on elements of the domain space instead of subsets of it. Section \ref{tailor} shows how to make a functional connection between core decreasing functions and nonnegative, nonincreasing functions on the half line. In Section \ref{decreasing} we define the level function and least decreasing majorant constructions in this general setting and the spaces of functions connected with them. The interpolation theory of these spaces is developed in Section \ref{interp}.

If $(U,\Sigma, \mu)$ is a $\sigma$-finite measure space and $\mathcal A\subseteq\Sigma$ we let $\sigma(\mathcal A)$ be the $\sigma$-ring generated by $\mathcal A$ and let $L(\mathcal A)$ be the collection of all $[-\oo,\oo]$-valued $\sigma(\mathcal A)$-measurable functions on $U$. In particular, $L(\Sigma)$ is the collection of $\Sigma$-measurable functions. The collection of nonnegative functions in $L(\mathcal A)$ is denoted $L^+(\mathcal A)$.
\begin{definition}  Let $(U,\Sigma, \mu)$ be a $\sigma$-finite measure space. 
\begin{enumerate}[label=(\alph*)]
\item An \emph{ordered core} of $(U,\Sigma, \mu)$ is a totally ordered subset $\mathcal A$ of $\Sigma$, containing the empty set, that consists of sets of finite $\mu$-measure. 
\item We say an ordered core $\mathcal A$ is \emph{$\sigma$-bounded} if there exists a countable subset $\mathcal A_0$ of $\mathcal A$ such that $\cup \mathcal A=\cup\mathcal A_0$ and we say $\mathcal A$ is \emph{full} if it is $\sigma$-bounded and $\cup \mathcal A=U$. (If $\mathcal A$ is $\sigma$-bounded, then  $\sigma(\mathcal A)$ is a $\sigma$-algebra over $\cup \mathcal A$.)
\item The relation $\le_{\mathcal A}$ on $U$ is defined by $u\le_{\mathcal A} v$ if for all $A\in\mathcal A$, $v\in A$ implies $u\in A$. We will drop the subscript $\mathcal A$ when there is no ambiguity.
\item A function $f:U\to[-\oo,\oo]$ is called \emph{decreasing} (relative to $\mathcal A$) for all $u,v\in U$,  $u\le_{\mathcal A}v$ implies $f(u)\ge f(v)$. A nonnegative, $\sigma(\mathcal A)$-measurable, decreasing function is called \emph{core decreasing}. The collection of core decreasing functions is denoted $L^\downarrow(\mathcal A)$.  
\end{enumerate}
\end{definition}

Observe that the relation ``$\le_{\mathcal A}$'' is reflexive and transitive but not antisymmetric in general. It is total, in the sense that for all $u,v\in U$, $u\le v$ or $v\le u$ or both.

See Section \ref{Examples} for examples of ordered cores and their order relations. As we will see in Example \ref{ordinal}, decreasing functions need not be $\mu$-measurable and even if they are $\mu$-measurable, they may not be $\sigma(\mathcal A)$-measurable. This is why we explicitly require $\sigma(\mathcal A)$-measurability in the collection $L^\downarrow(\mathcal A)$ of core decreasing functions.

\begin{enumerate}[label=, left=0pt]
\item
\item \textbf{Teasers:} Here is a look ahead at some consequences of the theory to be developed: Let $\mathcal A$ be a full ordered core on $(U,\Sigma, \mu)$. 
\item

\item\textbf{Level Functions:} For each $\mu$-measurable function $f$, there is a core decreasing function $f^o$ such that for all core decreasing functions $g$,
\[
\int_U f^og\,d\mu=\sup\bigg\{\int_U|f|h\,d\mu:h\in L^\downarrow(\mathcal A), \int_Ah\,d\mu\le\int_Ag\,d\mu\text{ for all } A\in\mathcal A\bigg\}.
\]
\item\textbf{D-type H\"older Inequalities:} If $1\le p\le\oo$, then $\|f^o\|_{L^p_\mu}\le\|f\|_{L^p_\mu}$ and for all core decreasing functions $g$,
\[
\int_U|f|g\,d\mu\le\|f^o\|_{L^p_\mu}\|g\|_{L^{p'}_\mu}.
\]
\item\textbf{Calder\'on Couples:} Let $\|f\|_{D^\oo_\mu}=\sup_{A\in\mathcal A}\frac1{\mu(A)}\int_A|f|\,d\mu$ and $D^\oo_\mu$ be the set of $f$ for which it is finite. Then $D^\oo_\mu$ is a Banach space, $(L^1_\mu,D^\oo_\mu)$ is a Calder\'on couple, and $Y$ is an exact interpolation space for the couple if and only if there is a universally rearrangement invariant space $X$ such that $\|f\|_Y=\|f^o\|_X$, for all $f$.
\end{enumerate}

\section{Notation and Background}\label{background}

We use $0\le \alpha_n\uparrow\alpha$ to indicate the limit of a nondecreasing sequence in $[0,\oo]$. Expressions that evaluate to $0/0$ will be taken to be $0$. The $\sgn$ function is defined on $\mathbb R$ by $\sgn(y)=|y|/y$ when $y\ne0$ and $\sgn(0)=0$. If $1\le p\le\oo$, and $\mu$ is a measure,  $L^p_\mu$ denotes the usual Lebesgue space of $\mu$-measurable functions.

\subsection{Universally rearrangement invariant spaces}\label{uri}
Let $(U,\Sigma,\mu)$ be a $\sigma$-finite measure space. The \emph{distribution function}, $\mu_f$, and \emph{rearrangement}, $f^*$, of $f\in L(\Sigma)$ are given by
\[
\mu_f(\delta)=\mu(\{u\in U:|f(u)|>\delta\})\quad\text{and}\quad
f^*(t)=\inf\{\delta>0:\mu_f(\delta)\le t\}
\]
for $\delta,t>0$. These take values in $[0,\oo]$.

Following \cite{Z}, a \emph{Banach function space} over $U$ is a real Banach space $X\subseteq L(\Sigma)$ that satisfies the following lattice property: If $f\in L(\Sigma)$, $g\in X$, and $|f|\le g$ $\mu$-a.e., then $f\in X$ and $\|f\|_X=\||f|\|_X\le\|g\|_X$. The space $X$ is \emph{saturated} if every $E\in\Sigma$ of positive measure has a subset $F\in\Sigma$ of positive measure such that $\chi_F\in X$. We say $X$ has the \emph{Fatou property} if whenever $0\le f_n\in X$ for each $n$, $f_n\uparrow f$ $\mu$-a.e., and $\lim_{n\to\oo}\|f_n\|_X<\oo$, it follows that $f\in X$ and $\|f_n\|_X\uparrow\|f\|_X$.

The \emph{associate space} $X'$ of $X$ is the Banach space of all $g\in L(\Sigma)$ such that 
\[
\|g\|_{X'}=\sup_{\|f\|_X\le1}\int|fg|\,d\mu<\oo.
\]
A \emph{universally rearrangement invariant} (u.r.i.) space over $U$ is a Banach function space over $U$ such that if $f\in L(\Sigma)$, $g\in X$ and $\int_0^tf^*\le\int_0^tg^*$ for all $t>0$, then $f\in X$ and $\|f\|_X\le\|g\|_X$. A nontrivial u.r.i. space is automatically saturated.

If $X$ is a saturated Banach function space over $U$, then $X'$ is a saturated Banach function space over $U$ with the Fatou property. If $X$ has the Fatou property, then $X''=X$ with identical norms. If $X$ is u.r.i., so is $X'$. See \cite{BS}, especially pages 64 and 90, for the close connection between u.r.i. spaces and rearrangement invariant spaces, but keep in mind that function spaces there are assumed to have additional properties, including the Fatou property. 

\subsection{Calder\'on couples}\label{oldCC}

A pair $(X_0,X_1)$ of Banach spaces is a \emph{compatible couple} if both spaces can be continuously embedded in a single Hausdorff topological vector space, thus making sense of the expressions $X_0\cap X_1$ and $X_0+X_1$. An \emph{admissible contraction} from $(X_0,X_1)$ to another compatible couple $(Y_0,Y_1)$, is a linear map $W$ on $X_0+X_1$ such that the restriction of $W$ to $X_j$ is a bounded map from $X_j$ to $Y_j$, of norm at most $1$, for $j=0,1$. A Banach space $X$ satisfying $X_0\cap X_1\subseteq X\subseteq X_0+X_1$ is called an \emph{exact interpolation space} for the couple $(X_0,X_1)$ if every admissible contraction from $(X_0,X_1)$ to itself maps each element of $X$ into $X$ with no increase in norm. 

The \emph{$K$-functional} for a compatible couple $(X_0,X_1)$ is given by
\[
K(f,t;X_0,X_1)=\inf(\|f_0\|_{X_0}+t\|f_1\|_{X_1}:f_0+f_1=f\},
\]
for $f\in X_0+X_1$ and $t>0$. We say $(X_0,X_1)$ is an \emph{exact Calder\'on couple} (or exact Calder\'on-Mityagin couple) if whenever $f,g\in X_0+X_1$ satisfy $K(f,t;X_0,X_1)\le K(g,t;X_0,X_1)$ for all $t>0$ there exists an admissible contraction from $(X_0,X_1)$ to itself that maps $g$ to $f$. Also, $(X_0,X_1)$ has \emph{divisibility constant $1$} if for all $f\in X_0+X_1$ and all nonnegative, concave functions $\omega_j$ on $[0,\oo)$, the conditions $\sum_{j=1}^\oo\omega_j(1)<\oo$ and $K(f,t;X_0,X_1)\le\sum_{j=1}^\oo\omega_j(t)$ for all $t\ge0$ imply that there exist $f_j\in X_0+X_1$ such that $K(f_j,t;X_0,X_1)\le\omega_j(t)$, for all $j$ and $t$, and $\sum_{j=1}^\oo f_j$ converges to $f$ in $X_0+X_1$. Here, $\|f\|_{X_0+X_1}=K(f,1;X_0,X_1)$. See \cite{BS} for the above definitions.

A Banach function space $\Phi$ of functions on the half line with measure $dt/t$ is a \emph{parameter of the $K$-method} if it contains the function $t\mapsto \min(1,t)$. If $(X_0,X_1)$ is a compatible couple, then $(X_0,X_1)_\Phi$ is the space of all $f\in X_0+X_1$ whose norm, $\|K(f,\cdot;X_0,X_1)\|_\Phi$, is finite. It is an exact interpolation space for $(X_0,X_1)$. If $(X_0,X_1)$ is an exact Calder\'on couple with divisibility constant $1$, then every exact interpolation space for $(X_0,X_1)$ is equal, with identical norms, to $(X_0,X_1)_\Phi$ for some parameter $\Phi$. See \cite{BK}*{Theorems 3.3.1 and 4.4.5 and Remark 4.4.4}.

For any $\sigma$-finite measure space $(P,\mathcal P,\rho)$, $K(f,t;L^1_\rho,L^\oo_\rho)=\int_0^tf^*$, $(L^1_\rho,L^\oo_\rho)$ is a Calder\'on couple with divisibility constant $1$, and the exact interpolation spaces of $(L^1_\rho,L^\oo_\rho)$ are exactly the u.r.i. spaces over $P$. See \cite{Ca}.

We will need a couple of standard (easy) results: If $W:(X_0,X_1)\to(Y_0,Y_1)$ is an admissible contraction then
\begin{equation}\label{T}
K(Wf,t;Y_0,Y_1)\le K(f,t;X_0,X_1).
\end{equation}
If $(X_0,X_1)$ is a compatible couple of Banach function spaces (over the same measure space,) then 
\begin{equation}\label{Kpos}
K(f,t;X_0,X_1)=\inf\{\|f_0\|_{X_0}+t\|f_1\|_{X_1}:f_0+f_1=|f|, 0\le f_0, 0\le f_1\}.
\end{equation}

\subsection{Spaces defined by nonincreasing functions}\label{old} Suppose $\lambda$ is a $\sigma$-finite, Borel measure on $[0,\oo)$ satisfying $\lambda([0,x])<\oo$ for each $x\in\mathbb R$. 
If $X$ is a u.r.i. space over $([0,\oo),\lambda)$, and $L^\downarrow$ denotes the collection of nonnegative, nonincreasing functions, let $X\!\!\downarrow$ be the space of $\lambda$-measurable $\varphi$ such that
\[
\|\varphi\|_{X\!\downarrow}=\sup\bigg\{\int_{[0,\oo)}|\varphi|\psi\,d\lambda: \psi\in L^\downarrow, \|\psi\|_{X'}\le1\bigg\}
\]
is finite. Then, for each $\lambda$-measurable $\varphi$ there exists a $\varphi^o\in L^\downarrow$, called the \emph{level function} of $\varphi$,  such that for all $\xi\in L^\downarrow$, 
\[
\int_{[0,\oo)}\varphi^o\xi\,d\lambda=\sup\bigg\{\int_{[0,\oo)}|\varphi|\psi\,d\lambda:\psi\in L^\downarrow, \int_{[0,x]}\psi\,d\lambda\le\int_{[0,x]}\xi\,d\lambda\text{ for }x\ge0\bigg\}
\]
and if $0\le \varphi_n\uparrow \varphi$ $\lambda$-a.e., then $\varphi_n^o\uparrow \varphi^o$ $\lambda$-a.e. Also, $L^1_\lambda\!\!\downarrow=L^1_\lambda$, with identical norms; $\|\varphi\|_{L_\lambda^\oo\!\downarrow}=\sup_{x\ge0}\frac1{\lambda([0,x])}\int_{[0,x]}|\varphi|\,d\lambda$;
$K(\varphi,t;L^1_\lambda,L^\oo_\lambda\!\!\downarrow)=\int_0^t(\varphi^o)^*$; and $(L^1_\lambda,L^\oo_\lambda\!\!\downarrow)$ is an exact Calder\'on couple with divisibility constant $1$.

For a $\lambda$-measurable $\psi$, define $\widetilde \psi(x)$ to be the essential supremum of $\psi$ on the interval $[x,\oo)$. Called the \emph{least decreasing majorant} of $\psi$, $\widetilde \psi$ is in $L^\downarrow$; if $\xi\in L^\downarrow$ and $\psi\le\xi$, then $\widetilde \psi\le \xi$; and
if $\psi_n\uparrow \psi$ $\lambda$-a.e. then $\widetilde {\psi_n}\uparrow\widetilde \psi$ $\lambda$-a.e. If $X$ is a Banach function space of $\lambda$-measurable functions that contains all characteristic functions of sets of finite measure, we let $\widetilde X$ be the space of functions for which $\|f\|_{\widetilde X}=\|\widetilde f\|_X$ is finite. Then $\widetilde{L^\oo_\lambda}=L^\oo_\lambda$, with identical norms; 
$K(\psi,t;\widetilde{L^1_\lambda},L^\oo_\lambda)=\int_0^t\big(\widetilde \psi\big)^*$; and
$(\widetilde{L^1_\lambda},L^\oo_\lambda)$ is an exact Calder\'on couple with divisibility constant $1$.

These results may be found in \cite{Stm}*{Proposition 1.5 and Theorem 2.3}, \cite{MS1}*{Corollary 3.9} and \cite{MS2}*{Lemma 3.2, Corollary 4.2 and Theorem 4.2}. In these references, results are stated for a measure on $\mathbb R$, not on $[0,\oo)$, but we identify $\lambda$ with a measure on $\mathbb R$ for which $\lambda((-\oo,0))=0$ so the results apply.

\section{Maps Between Ordered Cores}\label{Core Maps}

When measure spaces with ordered cores were introduced in \cite{Snf}, it was shown that well-behaved maps between measure spaces with $\sigma$-bounded ordered cores induce bounded linear operators of functions on these measure spaces. In this section we take a closer look at these operators and establish additional properties that we will require later.

Let $(P,\mathcal P, \rho)$ and $(T,\mathcal T,\tau)$ be $\sigma$-finite measure spaces and let $\mathcal A$ be a $\sigma$-bounded ordered core of $(P,\mathcal P, \rho)$. Let $c$ be a positive constant and $r:\mathcal A\to\mathcal T$ be an order-preserving map satisfying
\begin{equation}\label{bdd}
\tau(r(A)\setminus r(B))\le c\rho(A\setminus B)
\end{equation}
for all $A,B\in\mathcal A$. Note that $\{r(A)\setminus r(\emptyset):A\in\mathcal A\}$ is an ordered core of  $(T,\mathcal T,\tau)$.

\begin{theorem}\label{coremap} Let $V$ be the vector space of all $f\in L(\mathcal T)$ such that $f$ is integrable on $r(A)$ for each $A\in \mathcal A$. Then there is a map $R$ defined on $L^+(\mathcal T)\cup V$ such that for all $f\in L^+(\mathcal T)\cup V$:
\begin{enumerate}[label=(\alph*)]
\item\label{a} $Rf\in L(\mathcal A)$ and $Rf=0$ off $\cup\mathcal A$;
\item\label{b} if $f\in L^+(\mathcal T)$ then $Rf\in L^+(\mathcal A)$;
\item\label{c} if $f_n\in L^+(\mathcal T)$ for each $n$ and $f_n\uparrow f$ $\mu$-a.e., then $Rf_n\uparrow Rf$ $\rho$-a.e.;
\item\label{d} $R$ is linear on $V$ and it is additive and positive homogeneous on $L^+(\mathcal T)$;
\item\label{e} $|Rf|\le R|f|$ $\rho$-a.e.;
\item\label{f} for all $A,B\in\mathcal A$,
\[
\int_{B\setminus A}Rf\,d\rho=\int_{r (B)\setminus r (A)}f\,d\tau;
\]
\item\label{g} if $f\in L^1_\tau$, then $R f\in L^1_\rho$ and $\|R f\|_{L^1_\rho}\le\|f\|_{L^1_\tau}$; 
\item\label{h} if $f\in L^\infty_\tau$, then $R f\in L^\oo_\rho$ and $\|R f\|_{L^\infty_\rho}\le c\|f\|_{L^\infty_\tau}$;
\item\label{i} if $c=1$, $r(\emptyset)=\emptyset$, $A\in\mathcal A$ and $\rho(A)=\tau(r(A))$, then $R\chi_{r(A)}=\chi_A$ $\rho$-a.e.
\item\label{j} if $c=1$, $r(\emptyset)=\emptyset$, $\rho(A)=\tau(r(A))$ for all $A\in\mathcal A$, $f\in L^+(\mathcal T)$ and $g\in L^+(r(\mathcal A))$, then $R(fg)=RfRg$ $\rho$-a.e. 
\end{enumerate}
\end{theorem}

\begin{proof} \textbf{The original map.} From Theorem 4.6 of \cite{Snf} there is a linear map $R:L^1_\tau+L^\oo_\tau\to L^1_\rho+L^\oo_\rho$ that maps nonnegative functions to nonnegative functions and satisfies \ref{f}, \ref{g}, and \ref{h} for each $f\in L^1_\tau+L^\oo_\tau$. In the proof of that theorem it is shown that \ref{a} is satisfied for each nonnegative $f\in L^1_\tau+L^\oo_\tau$.

The following observation will be needed: If $f\in L^1_\tau+L^\oo_\tau$ and $(f_n)$ is an increasing sequence of nonnegative functions in $L^1_\tau+L^\oo_\tau$ that converges  to $f$ pointwise $\tau$-a.e., then $Rf_n$ converges to $Rf$ pointwise $\rho$-a.e. To see this, fix $A\in\mathcal A$. Since $\rho(A)<\oo$ and $Rf\in L^1_\rho+L^\oo_\rho$, $\int_A Rf\,d\rho<\oo$. By \ref{f}, and the Monotone Convergence Theorem, 
\[
\int_A Rf\,d\rho=\int_{r(A)\setminus r(\emptyset)} f\,d\tau=\lim_{n\to\oo}\int_{r(A)\setminus r(\emptyset)} f_n\,d\tau=\lim_{n\to\oo}\int_A Rf_n\,d\rho.
\]
The positivity of $R$ on $L^1_\tau+L^\oo_\tau$ ensures that $Rf_n$ is an increasing sequence bounded above by $Rf$. So another application of the Monotone Convergence theorem yields
\[
\int_A Rf\,d\rho
=\int_A \lim_{n\to\oo}Rf_n\,d\rho\le\int_A Rf\,d\rho.
\]
Thus, $Rf=\lim_{n\to\oo}Rf_n$ $\rho$-a.e. on $A$. By the $\sigma$-boundedness of $\mathcal A$,  $Rf=\lim_{n\to\oo}Rf_n$ $\rho$-a.e. on $\cup\mathcal A$. But both $Rf$ and $Rf_n$ are zero off $\cup\mathcal A$ so $Rf_n\to Rf$ pointwise $\rho$-a.e. on $P$.

\textbf{Extension to nonnegative functions.} With this in hand, we may define $R$ on $L^+(\mathcal T)$. For each $f\in L^+(\mathcal T)$ let $Rf$ be the pointwise limit of $Rf_n$, where $(f_n)$ is an increasing sequence in $L^+(\mathcal T)\cap L^1_\tau$ that converges pointwise to $f$. The previous observation ensures that this $R$ agrees with the original $R$ whenever both are defined.

Since $\tau$ is $\sigma$-finite, each $f\in L^+(\mathcal T)$ can be expressed as the limit of such a sequence and since $R$ maps nonnegative functions to nonnegative functions, $Rf_n$ is an increasing sequence in $L^+(\mathcal T)$ and hence has a (finite or infinite) pointwise limit. It remains to show that the definition of $Rf$ on $L^+(\mathcal T)$ does not depend on the choice of $(f_n)$. Fix $f\in L^+(\mathcal T)$ and let $(f_n)$ and $(g_n)$ be two increasing sequences in $L^+(\mathcal T)\cap L^1_\tau$ that converge pointwise to $f$. For each fixed $m$, $h_n=\min(f_n,g_m)$ defines an increasing sequence in $L^+(\mathcal T)\cap L^1_\tau$ that converges pointwise to $g_m$.  By the Monotone Convergence Theorem, $h_n\to g_m$ in the space $L^1_\tau$. By \ref{g}, $Rh_n\to Rg_m$ in $L^1_\rho$. But $(Rh_n)$ is a increasing sequence so the Monotone Convergence Theorem shows that  $Rh_n$ converges to $Rg_m$ pointwise $\rho$-a.e. Since $h_n\le f_n$, $Rh_n\le Rf_n$ and therefore $Rg_m=\lim_{n\to\oo} Rh_n\le\lim_{n\to\oo}Rf_n$ for each $m$. Letting $m\to\oo$ we get $\lim_{m\to\oo}Rg_m\le\lim_{n\to\oo}Rf_n$. Reversing the roles of $f$ and $g$ gives the opposite inequality as well. Thus $Rf$ is well defined as a map from $L^+(\mathcal T)$ to $L^+(\mathcal T)$.

Since \ref{a} and \ref{b} are preserved by limits of sequences, they hold for all $f\in L^+(\mathcal T)$. 

To see that \ref{c} also holds, fix $f\in L^+(\mathcal T)$ and let $(f_n)$ be an increasing sequence in $L^+(\mathcal T)$ that converges pointwise to $f$. Take $(g_m)$ to be an increasing sequence in $L^+(\mathcal T)\cap L^1_\tau$ that converges pointwise to $f$. By definition, $Rg_m$ converges pointwise to $Rf$. For each fixed $m$, $h_n=\min(f_n,g_m)$ defines an increasing sequence in $L^+(\mathcal T)\cap L^1_\tau$ that converges pointwise to $g_m$ so $Rh_n$ converges pointwise to $Rg_m$. For each $n$, $h_n\le f_n$ so $Rh_n\le Rf_n\le\lim_{n\to\oo}f_n$. It follows that $Rg_m\le\lim_{n\to\oo}Rf_n$. Letting $m\to\oo$, we get
\[
Rf=\lim_{m\to\oo} Rg_m\le\lim_{n\to\oo}Rf_n\le Rf.
\] 

For \ref{d}, suppose $f,g\in L^+(\mathcal T)$ and $\alpha\in[0,\oo)$. Let $(f_n)$ and $(g_n)$ be increasing sequences in $L^+(\mathcal T)\cap L^1_\tau$ that converge pointwise to $f$ and $g$ respectively. Then $(\alpha f_n)$ and $(f_n+g_n)$ are increasing sequences in $L^+(\mathcal T)\cap L^1_\tau$ that converge pointwise to $\alpha f$ and $f+g$ respectively. Since $R$ is linear on $L^1_\tau$, $R(\alpha f)=\lim_{n\to\oo}R(\alpha f_n)=\alpha \lim_{n\to\oo}R(f_n)=\alpha Rf
$ and $R(f+g)=\lim_{n\to\oo}R(f_n+g_n)=\lim_{n\to\oo}R(f_n)+R(g_n)=Rf+Rg$. 

In this case, \ref{e} is trivial, \ref{f} is a consequence of the Monotone Convergence Theorem and \ref{g} and \ref{h} are unchanged.

\textbf{Extension to the vector space.} Next we define $R$ on $V$. For each $f\in V$, define $Rf$ to be $Rf^+-Rf^-$, where $f^+=(|f|+f)/2$ and $f^-=(|f|-f)/2$ as usual. Linearity of the original $R$ implies that this $R$ agrees with the original $R$ whenever both are defined.

Since $f$ is integrable on $r(A)$ for each $A\in\mathcal A$, both $f^+$ and $f^-$ are integrable on each $r(A)$. But $f^+,f^-\in L^+_\tau$ so \ref{f} implies that $Rf^+$ and $Rf^-$ are finite $\rho$-a.e. on each $A\in \mathcal A$. The core $\mathcal A$ is assumed to be $\sigma$-bounded so $Rf^+$ and $Rf^-$ are finite $\rho$-a.e. on $\cup\mathcal A$. By \ref{a}, $Rf^+$ and $Rf^-$ are zero off $\cup\mathcal A$ so the difference $Rf^+-Rf^-$ is defined $\rho$-a.e. on $P$.

Taking differences preserves \ref{a} so it remains valid on $V$; \ref{b} and \ref{c} involve only nonnegative functions; and 
\ref{d} is readily extended from $L^+(\mathcal T)$ to $V$ by applying $R$ to the identities, 
\begin{align*}
(\alpha f)^++\alpha f^-&=(\alpha f)^-+\alpha f^+\quad\mbox{for $\alpha\ge0$},\\
(\alpha f)^++(-\alpha) f^+&=(\alpha f)^-+(-\alpha) f^-\quad\mbox{for $\alpha<0$},\\
(f+g)^++f^-+g^-&=(f+g)^-+f^++g^+.
\end{align*}

For \ref{e},
\[
|Rf|=|Rf^+-Rf^-|\le Rf^++Rf^-=R(f^++f^-)=R|f|.
\]
The definition of $V$ and linearity of the integral gives \ref{f}. Again, \ref{g} and \ref{h} are unchanged.

\textbf{The last two properties.} Now we turn our attention to \ref{i} and \ref{j}. If $c=1$, $r(\emptyset)=\emptyset$, and $A\in\mathcal A$, taking $f=\chi_{r(A)}$ in \ref{f}, \ref{g}, and \ref{h}  give
\[
\int_A R\chi_{r(A)}\,d\rho =\tau(r(A)), \quad \int_P R\chi_{r(A)}\,d\rho \le\tau(r(A)),\quad\mbox{and}\quad R\chi_{r(A)}\le1\  \rho\text{-a.e.}
\]
Since $\tau(r(A))<\oo$, it follows that $R\chi_{r(A)}\le\chi_A$ $\rho$-a.e. But $\int_P\chi_A- R\chi_{r(A)}\,d\rho=\rho(A)-\tau(r(A))=0$, which implies \ref{i}.

For \ref{j}, suppose $c=1$, $r(\emptyset)=\emptyset$, and $\rho(A)=\tau(r(A))$ for all $A\in \mathcal A$. First suppose $f\in L^+(\mathcal T)$ is bounded above. By \ref{h}, $Rf$ is also bounded above $\rho$-a.e. Fix a $B\in\mathcal A$. Then $\rho(B)<\oo$ and for every $E\in\mathcal T$, \ref{b} and \ref{h} show that $R(f\chi_E)\le Rf$ and $R\chi_E\le 1$ $\rho$-a.e. These observations, combined with \ref{d} and \ref{c}, show that the maps
\[
\eta(E)=\int_BRfR\chi_E\,d\rho\quad\mbox{and}\quad \zeta(E)=\int_BR(f\chi_E)\,d\rho
\]
define finite measures on the measurable space $(T,\mathcal T)$. If $E=r(A)$ for some $A\in\mathcal A$, and we let $C$ denote the smaller of $A$ and $B$, then $r(C)$ is the smaller of $r(A)$ and $r(B)$ so, using \ref{i} and \ref{f}, we get
\[
\eta(r(A))=\int_BRfR\chi_{r(A)}\,d\rho=\int_B (Rf)\chi_A\,d\rho=\int_CRf\,d\rho
\]
and
\[
\zeta(r(A))=\int_BR(f\chi_{r(A)})\,d\rho=\int_{r(B)}f\chi_{r(A)}\,d\tau
=\int_{r(C)}f\,d\tau=\int_CRf\,d\rho.
\]
Thus, $\eta(r(A))=\zeta(r(A))$ for all $A\in\mathcal A$. Taking differences, we see that the measures $\eta$ and $\zeta$ coincide on the semiring $\{r(A_1)\setminus r(A_2):A_1,A_2\in\mathcal A\}$, see Lemma 4.3 of \cite{Snf}. They therefore coincide on the $\sigma$-ring generated by the semiring, see \cite{RF}*{Corollary 14 on page 357}, which is $\sigma(r(\mathcal A))$. So for each $E\in\sigma(r(\mathcal A))$,  
\[
\int_BRfR\chi_E\,d\rho=\int_BR(f\chi_E)\,d\rho.
\]
By \ref{a}, $Rf$, $R\chi_E$, and $R(f\chi_E)$ are $\sigma(\mathcal A)$-measurable on $\cup\mathcal A$ and zero off $\cup\mathcal A$. It follows that 
\[
R(f\chi_E)=RfR\chi_E
\]
$\rho$-a.e.

If $g\in L^+(r(\mathcal A))$, it can be expressed as the pointwise limit of an increasing sequence of finite linear combinations, with positive coefficients, of characteristic functions of sets in $\sigma(r(\mathcal A))$. So, by \ref{d} and \ref{c}, $R(fg)=RfRg$.

Each $f\in L^+(\mathcal T)$ can be expressed as an increasing sequence of bounded functions in $L^+(\mathcal T)$ so one more application of \ref{c} completes the proof.
\end{proof}

\section{Core Decreasing Functions}\label{Core Functions}

Let $(U,\Sigma,\mu)$ be a $\sigma$-finite measure space and let $A$ be an ordered core.

Many different ordered cores may give rise to the same order on elements and generate the same $\sigma$-ring. For our purposes, it will be convenient to enrich our ordered core $\mathcal A$ by adding in as many additional sets as we can while ensuring that neither the order $\le_{\mathcal A}$ nor the $\sigma$-ring $\sigma(\mathcal A)$ is altered. This is done in Theorem \ref{M}, below. First we show that these additional sets can be characterized in three different ways.
\begin{lemma}\label{3fae} Suppose $M\in\sigma(\mathcal A)$ and $\mu(M)<\oo$. The following are equivalent:
\begin{enumerate}[label=(\alph*)]
\item\label{sC} There is a countable subset $\mathcal C$ of $\mathcal A$ such that $M=\cup\mathcal C$ or $M=\cap\mathcal C$;
\item\label{C} There is a subset $\mathcal C$ of $\mathcal A$ such that $M=\cup\mathcal C$ or $M=\cap\mathcal C$;
\item\label{rs} For all $u,v\in U$, if $v\in M$ and $u\le_\mathcal A v$, then $u\in M$.
\end{enumerate}
In the above, if $\mathcal C=\emptyset$ we take $\cup\mathcal C=\emptyset$ and $\cap\mathcal C=\cup\mathcal A$.
\end{lemma}
\begin{proof} If \ref{sC} holds then so does \ref{C}. Next suppose that \ref{C} holds for a subset $\mathcal C$, that $v\in M$ and that $u\le_\mathcal A v$. We show that $u\in M$ in two cases: If $M=\cup \mathcal C$, then for some $A\in\mathcal C$, $v\in A$ and hence $u\in A$. Thus $u\in M$. If $M=\cap \mathcal C$, then for all $A\in\mathcal C$, $v\in A$ and hence $u\in A$. Thus $u\in M$. This proves \ref{rs}.

Finally, we suppose that \ref{rs} holds. First observe that $\mathcal A\cup\{M\}$ remains totally ordered: If $A\in\mathcal A$ and $M\not\subseteq A$, choose $v\in M\setminus A$. If $u\in A$, then $v\not\le_\mathcal A u$ so $u\le_\mathcal A v$ and we have $u\in M$. Thus $A\subseteq M$.

Since $M\in\sigma(\mathcal A)$ we may choose a countable subset $\mathcal A_0$ of $\mathcal A$ such that $M\in\sigma(\mathcal A_0)$. (This is \cite{Ha}*{Theorem D on page 21}, but readily follows from the observation that the union of all $\sigma$-rings generated by countable subsets of $\mathcal A$ is itself a $\sigma$-ring.) Now let $U_0=\cup\mathcal A_0$, 
\[
\mathcal L=\{L\in\mathcal A_0:L\subseteq M\},\quad \mathcal N=\{N\in\mathcal A_0:M\subseteq N\},\quad C=(\cap\mathcal N)\setminus(\cup\mathcal L),
\]
and
\[
\mathcal K=\{A\in\sigma(\mathcal A_0):C\subseteq U_0\setminus A\text{ or }C\subseteq A\}.
\]
Since $\mathcal A_0\cup\{M\}$ is totally ordered, $\mathcal A_0=\mathcal L\cup\mathcal N$. If $L\in\mathcal L$, $C\subseteq U_0\setminus L$. If $N\in\mathcal N$, then $C\subseteq N$. Therefore, $\mathcal A_0\subseteq\mathcal K$. To see that $\sigma(A_0)\subseteq\mathcal K$ it is enough to show that $\mathcal K$ is a $\sigma$-algebra of subsets of $U_0$. Clearly, $\emptyset\in \mathcal K$, $U_0\in\mathcal K$, and $\mathcal K$ is closed under complementation. If $A_n\in\mathcal K$ for $n=1,2,\dots$, then either $C\subseteq A_n$ for all $n$ or $C\subseteq U_0\setminus A_n$ for some $n$. In the first case $C\subseteq \cap_n A_n$ and in the second case $C\subseteq \cup_n (U_0\setminus A_n)=U_0\setminus\cap_n A_n$. Thus $\mathcal K$ is closed under countable intersections. This shows that $\mathcal K$ is a $\sigma$-algebra. But $M\in\sigma(\mathcal A_0)$, so $M\in\mathcal K$. It follows that $C\subseteq U_0\setminus M$ or $C\subseteq M$ so either $M=\cup\mathcal L$ or $M=\cap\mathcal N$. Since both $\mathcal L$ and $\mathcal N$ are countable subsets of $\mathcal A$, this establishes \ref{sC} and completes the proof.
\end{proof}

Now we are ready to produce the enriched core. Although the image of $\mathcal A$ under $\mu$, i.e. $\mu(\mathcal A)\equiv\{\mu(A):A\in\mathcal A\}$, may not be a closed subset of $[0,\oo)$, the image of the enriched core always is. 

\begin{theorem}\label{M} Let $\mathcal M$ be the collection of all $M\in\sigma(\mathcal A)$, of finite measure, for which one, and hence all, of \ref{sC}, \ref{C}, and \ref{rs} of Lemma \ref{3fae} holds. Then $\mathcal M$ is an ordered core of $U$, $\mathcal A\subseteq\mathcal M$, $\sigma(\mathcal M)=\sigma(\mathcal A)$, and the relations $\le_{\mathcal M}$ and $\le_{\mathcal A}$ coincide. If $\mathcal A$ is $\sigma$-bounded, so is $\mathcal M$. If $\mathcal A$ is full, so is $\mathcal M$. In addition, $\mathcal M$ is closed under (nonempty) countable intersections and under countable unions provided the result has finite measure. Finally, $\mu(\mathcal M)$ is the closure in $[0,\oo)$ of $\mu(\mathcal A)$.
\end{theorem}
\begin{proof} It is immediate that $\mathcal A\subseteq\mathcal M$ so $\emptyset\in \mathcal M$ and $\sigma(\mathcal A)\subseteq\sigma(\mathcal M)$. By \ref{sC}, $\mathcal M\subseteq\sigma(\mathcal A)$ and we see that $\sigma(\mathcal A)=\sigma(\mathcal M)$. Since every element of $\sigma(\mathcal A)$ is a subset of $\cup\mathcal A$, $\cup\mathcal M=\cup\mathcal A$. Thus, if $\mathcal A$ is $\sigma$-bounded, so is $\mathcal M$ and if $\mathcal A$ is full, so is $\mathcal M$.

Suppose $M,N\in\mathcal M$ such that $N\not\subseteq M$ and choose $v\in N\setminus M$. If $u\in M$ then \ref{rs} implies $v\not\le_{\mathcal A}u$. Thus $u\le_{\mathcal A}v$ and we conclude that $u\in N$. This shows $M\subseteq N$ so $\mathcal M$ is totally ordered. Therefore, $\mathcal M$ is an ordered core of $U$.

Since $\mathcal A\subseteq\mathcal M$, $u\le_{\mathcal A}v$ holds whenever $u\le_{\mathcal M}v$ does. On the other hand, if $u\le_{\mathcal A}v$ holds and $M\in\mathcal M$ with $v\in M$, then \ref{rs} shows that $u\in M$. So $u\le_{\mathcal M}v$ holds. Thus the relations $\le_{\mathcal M}$ and $\le_{\mathcal A}$ coincide.

Let $M_n\in\mathcal M$ for all $n$. If $M=\cap_n M_n$, then $M\in\sigma(\mathcal A)$ and $\mu(M)<\oo$. If $v\in M$ and $u\le_{\mathcal A}v$, then $v\in M_n$ for all $n$ so $u\in M_n$ for all $n$ and therefore $u\in M$. Thus $M\in\mathcal M$.  If $M=\cup_n M_n$, then $M\in\sigma(\mathcal A)$. If $v\in M$ and $u\le_{\mathcal A}v$, then $v\in M_n$ for some $n$ so $u\in M_n$ for some $n$ and therefore $u\in M$. Thus $M\in\mathcal M$ provided $\mu(M)<\oo$. 

It remains to show that $\mu(\mathcal M)$ is the closure of $\mu(\mathcal A)$ in $[0,\oo)$. First suppose that $x$ is in the closure of $\mu(A)$. If $x\in\mu(\mathcal A)$ then $x\in\mu(\mathcal M)$ because $\mathcal A\subseteq \mathcal M$. If $x\notin\mu(\mathcal A)$ then $x$ can be expressed as the limit of a sequence $(x_n)$ in $\mu(\mathcal A)$ that is either strictly increasing or strictly decreasing. For each $n$ choose $A_n\in\mathcal A$ such that $x_n=\mu(A_n)$. If $x_n$ is strictly increasing, the total ordering of $\mathcal A$ ensures that $A_1\subseteq A_2\subseteq \dots$ so $x=\lim_{n\to\oo}\mu(A_n)=\mu(\cup_n A_n)\in\mu(\mathcal M)$. If $x_n$ is strictly decreasing, the total ordering of $\mathcal A$ ensures that $A_1\supseteq A_2\supseteq \dots$ and, since $\mu(A_1)<\oo$, we get $x=\lim_{n\to\oo}\mu(A_n)=\mu(\cap_n A_n)\in\mu(\mathcal M)$.

Conversely, suppose $x\in\mu(\mathcal M)$. Then $x=\mu(M)$ for some $M\in\mathcal M$. By \ref{sC}, we may choose $A_1, A_2,\dots\in\mathcal A$ such that either $M=\cup_n A_n$ or $M=\cap_nA_n$. Since $\mathcal A$ is totally ordered it is trivially closed under finite unions and finite intersections. Thus, $x=\lim_{n\to\oo}\mu(A_1\cup\dots\cup A_n)$ or $x=\lim_{n\to\oo}\mu(A_1\cap\dots\cap A_n)$. Each of these limits is in the closure of $\mu(\mathcal A)$. This completes the proof.
\end{proof}

We can characterize the collection $L^\downarrow(\mathcal A)=L^\downarrow(\mathcal M)$ as increasing limits of simple functions based on sets in the enriched core $\mathcal M$ just constructed.

\begin{lemma}\label{approx} Suppose $\mathcal A$ is $\sigma$-bounded. Let $f:U\to[0,\oo)$. Then $f\in L^\downarrow(\mathcal A)$ if and only if $f$ is the pointwise limit of an increasing sequence of functions of the form 
\[
\sum_{k=1}^K\alpha_k\chi_{M_k},
\]
where $\alpha_k>0$ and $M_k\in\mathcal M$ for each $k$.
\end{lemma}
		\begin{proof} If $M\in\mathcal M$, then $M\in\sigma(\mathcal M)=\sigma(\mathcal A)$ so $\chi_M\in L^+(\mathcal A)$. To see that $\chi_M$ is core decreasing, let $u\le v$. If $v\notin M$, $\chi_M(v)\le\chi_M(u)$ holds trivially and if $v\in M$ then Lemma \ref{3fae}\ref{rs} shows $u\in M$ so again $\chi_M(v)\le\chi_M(u)$.
 
It is routine to verify that $L^\downarrow(\mathcal A)$ is closed under the formation of finite linear combinations with positive coefficients and also under limits of increasing sequences whose limits are finite $\mu$-a.e. Thus, if $f$ is the pointwise limit of an increasing sequence of functions of the specified form, then  $f\in L^\downarrow(\mathcal A)$.

For the converse, suppose $f\in L^\downarrow(\mathcal A)$ and apply the $\sigma$-boundedness of $\mathcal A$ to choose an increasing sequence $A_n\in\mathcal A$ such that $\cup_n A_n=\cup\mathcal A$. For each integer $n>0$, set 
\[
M_{n,k}=\{u\in A_n:f(u)\ge k2^{-n}\}, \quad k=1,2,\dots,n2^n.
\]
Since $f$ is  $\sigma(\mathcal A)$-measurable, $M_{n,k}\in\sigma(\mathcal A)$. It has finite $\mu$-measure because $A_n$ does, and it satisfies \ref{3fae}\ref{rs} because $f$ is core decreasing. Thus, $M_{n,k}\in\mathcal M$. Define $f_n$ by
\[
f_n(u)=\sum_{k=1}^{n2^n}2^{-n}\chi_{M_{n,k}}(u)=2^{-n}\lfloor2^n\min(f(u),n)\rfloor\chi_{A_n}(u),
\]
where $\lfloor y\rfloor$ is the greatest integer less than or equal to $y$. To see that the two expressions are equal, suppose the right-hand side evaluates to $k_02^{-n}>0$. Then $u\in A_n$ and $k_0\le n2^n$. Also, $u\in M_{n,k}$ if and only if $k=1,2,\dots,k_0$. So the left-hand side also evaluates to $k_02^{-n}$.

The first expression shows that each $f_n$ is of the desired form. The second shows that $f_n$ increases pointwise to $f$: Fix $u\in U$. The inequality $2\lfloor y\rfloor\le\lfloor 2y\rfloor$, $y\ge0$, shows that $f_n(u)\le f_{n+1}(u)$ and the inequality $\lfloor y\rfloor\le y<\lfloor y\rfloor+1$ shows that $f_n(u)\le
f(u)< f_n(u)+2^{-n}$ for sufficiently large $n$.
\end{proof}

We record a simple consequence of the previous lemma for future reference. 

\begin{corollary}\label{M2g} Let $f,h\in L^+(\Sigma)$ and $g\in L^\downarrow(\mathcal A)$. If $\int_Mf\,d\mu\le\int_Mh\,d\mu$ for all $M\in\mathcal M$, then $\int_Ufg\,d\mu\le\int_Uhg\,d\mu$.
\end{corollary}

\section{Examples}\label{Examples}

Using subsets of the $\sigma$-algebra to carry the order in a measure space gives a great deal of flexibility when recognizing a class of functions that behaves like a class of decreasing functions. Here we offer a variety of examples.

To begin, we give an ordered core of Borel subsets of $\mathbb R$ that give rise to the usual order there. So any results proved for general ordered cores apply to the known cases.

\begin{example} Let $\mu$ be a Borel measure on $\mathbb R$ such that $\mu((-\oo,x])<\oo$ for each $x\in\mathbb R$. Take $\mathcal A=\{\emptyset\}\cup\{(-\oo,x]:x\in\mathbb R\}$. Then $\mathcal A$ is a full ordered core, $\sigma(\mathcal A)$ is the Borel $\sigma$-algebra, $x\le_{\mathcal A}y$ is the usual order on $\mathbb R$, and the core decreasing functions are just the usual nonnegative, nonincreasing functions. The enriched core is $\mathcal M=\{\emptyset\}\cup\{(-\oo,x), (-\oo,x]:x\in\mathbb R\}$.
\end{example}

In $\mathbb R^n$ we introduce a notion of order built on balls centred at zero, with radii in a fixed closed set. 

\begin{example} Let $U=\mathbb R^n$ with Lebesgue measure, fix a subset $S$ of $[0,\oo)$, let $B_s$ be the open ball of radius $s$ centred at zero, and take $\mathcal A=\{\emptyset\}\cup\{B_s:s\in S\}$. Then $\mathcal A$ is a $\sigma$-bounded ordered core, it is full if and only if $S$ is unbounded, $u\le_{\mathcal A}v$ if and only if $(|v|,|u|]\cap S=\emptyset$, and core decreasing functions are nonnegative, radially decreasing functions that are constant on annuli corresponding to components of the complement of $S$. 
\end{example}

In a metric space, the functions that decrease as their distance from a fixed subset increases behave like a class of decreasing functions. The structure of the class varies considerably with the choice of the fixed subset.
\begin{example} Let $\mu$ be a finite Borel measure on a metric space $U$. Fix a nonempty subset $U_0$ and take $\mathcal A=\{\emptyset, U_0\}\cup\{\{u\in U:\operatorname{dist}(u,U_0)<s\},s>0\}$. Then $\mathcal A$ is a full ordered core and $u\le_{\mathcal A}v$ depends heavily on the choice $U_0$. Core decreasing functions are those that decrease as the distance to $U_0$ increases, they will be constant on sets of fixed distance to $U_0$.  Interesting choices include taking $U_0$ be to be $(-1,1)$ on the $x$-axis when $U$ is the unit sphere in $\mathbb R^2$; taking $U_0=\mathbb Z^n$ when $U=\mathbb R^n$ with a finite measure; taking $U_0$ to be the boundary when $U$ is a Riemannian manifold with boundary.
\end{example}

The next example is really two examples. Both are based on the $\sigma$-algebra of countable and co-countable sets. Both use the total order on ordinals to give an order on elements of the set. But the choice of ordered core introduces a ``phantom point'' by including an uncountable totally ordered set that accumulates from both sides even though there is no point in the set to accumulate to. These examples show that order carried on sets can differ substantially from order carried on points and that the class of core decreasing functions can behave quite unlike typical decreasing functions.
\begin{example}\label{ordinal} If $S$ is a set, let $\sigma_{cc}(S)$ be the $\sigma$-algebra of countable subsets of $S$ together with their complements.

Let $U$ and $V$ be disjoint copies of $\omega_1$, the uncountable collection of all countable ordinals ordered by inclusion. Their smallest elements will be denoted $0_U$ and $0_V$, respectively. 
We will work in the disjoint union $U\mathbin{\dot\cup} V$ and express subsets as $E\mathbin{\dot\cup} F$, where $E\subseteq U$, $F\subseteq V$. Let 
\[
\Sigma_0=\sigma_{cc}(U\mathbin{\dot\cup} V)\quad\text{and}\quad\Sigma_1=\{E\mathbin{\dot\cup}F:E\in\sigma_{cc}(U), F\in\sigma_{cc}(V)\}.
\]
Both are $\sigma$-algebras over $U\mathbin{\dot\cup} V$ and $\Sigma_0\subseteq\Sigma_1$. Define $\mu_1$ on $\Sigma_1$ by setting 
\[
\mu_1(E\mathbin{\dot\cup} F)=\delta_{0_U}(E\mathbin{\dot\cup} F)+\delta_{0_V}(E\mathbin{\dot\cup} F)+\begin{cases}
0,&E,F\text{ countable};\\
1,&U\setminus E,F\text{ countable};\\
2,&E,V\setminus F\text{ countable};\\
3,&U\setminus E,V\setminus F\text{ countable};\\
\end{cases}
\]
and let $\mu_0$ be the restriction of $\mu_1$ to $\Sigma_0$. (Here $\delta_{0_U}$ and $\delta_{0_V}$ are Dirac measures at $0_U$ and $0_V$, respectively.) Then $\mu_0$ and $\mu _1$ are finite, complete measures on $\Sigma_0$ and $\Sigma_1$, respectively. Note that $\mu_0$ and $\mu_1$ have the same null sets, namely, the countable subsets of $U\mathbin{\dot\cup} V$ that don't include $0_U$ or $0_V$.

We introduce an ordered core that preserves the order on $U$, reverses the order on $V$, and makes every element of $U$ less than every element of $V$. For each $x\in U$ and $y\in V$, set $U_x=\{u\in U:u\prec x\}$ and $V_y=\{v\in V:v\prec y\}$. Here ``$\prec$'' is the strict order on ordinals. (Technically, $U_x=x$, $V_y=y$, and ``$\prec$'' is just ``$\subsetneq$''; the redundant notation is introduced to avoid confusion due to the definition of ordinals as sets of previous ordinals.) Note that both $U_x$ and $V_y$ are countable. Set
\[
\mathcal A=\{U_x\mathbin{\dot\cup}\emptyset:x\in U\}\cup\{U\mathbin{\dot\cup}(V\setminus V_y):y\in V\}.
\]
Then $\mathcal A$ is an ordered core of both $(U\mathbin{\dot\cup} V,\Sigma_0,\mu_0)$ and $(U\mathbin{\dot\cup} V,\Sigma_1,\mu_1)$. Note that $U\mathbin{\dot\cup} V\in\mathcal A$ so $\mathcal A$ is trivially a full ordered core. One readily shows that $\sigma(\mathcal A)=\Sigma_0$.

\begin{enumerate}[label=$\bullet$, left=0pt]
\item
The function $f=\chi_{U\mathbin{\dot\cup}\emptyset}$ is decreasing relative to $\mathcal A$. But $f$ is not $\Sigma_0$-measurable, so it is not core decreasing. So on $(U\mathbin{\dot\cup} V,\Sigma_0,\mu_0)$ with core $\mathcal A$ there is a nonmeasurable, decreasing function.
\item
The same function $f$ is $\Sigma_1$-measurable, but no function that agrees with $f$ $\mu_1$-a.e. is $\Sigma_0$-measurable. So on $(U\mathbin{\dot\cup} V,\Sigma_1,\mu_1)$ with core $\mathcal A$ there is a measurable, decreasing function that is not $\sigma(\mathcal A)$-measurable.
\end{enumerate}
Let $g=1-\chi_{\{0_V\}}$. Then $g$ is core decreasing in both measure spaces. For each $w\in U\mathbin{\dot\cup} V$, let $G(w)$ be the essential supremum of $g$ on $W=\{\bar w\in U\mathbin{\dot\cup} V:w\le_{\mathcal A}\bar w\}$. (Since $\mu_0$ and $\mu_1$ have the same null sets, the definition of $G$ is the same for both.) If $w\in U$, then $g$ takes the value $1$ on the co-countable set $(U\setminus U_w)\mathbin{\dot\cup}(V\setminus\{0_V\})$, which has positive $\mu_0$-measure and is contained in $W$. Thus $G(w)=1$. If $w\in V$, then $W=\emptyset\mathbin{\dot\cup}(\{w\}\cup V_w)$, a countable set. The value of $g$ is $1$ on $W\setminus\{0_V\}$, which has measure zero, and $0$ on $\{0_V\}$, which has measure $1$. Thus $G(w)=0$ and we have $G=f$. Note that on $\emptyset\mathbin{\dot\cup}(V\setminus\{0_V\})$, a set of positive $\mu_1$-measure, $f=0$ and $g=1$.
\begin{enumerate}[resume, label=$\bullet$, left=0pt]
\item So the essential supremum construction analogous to the one used to define the least decreasing majorant in Section \ref{old}, when applied to the core decreasing function $g$, produces the $\mu_0$-nonmeasurable function $f$. The function $f$ is $\mu_1$-measurable, but it is not core decreasing and it is not a majorant of $g$.
\end{enumerate}
See Lemma \ref{lcdm} for a different approach to proving the existence of a least core decreasing majorant.
\end{example}

In the previous examples, the core has been based on the structure of the underlying measure space. Here is an example, on a general measure space, where the ordered core is defined in terms of a single fixed function. Building an ordered core in this way recovers the notion of ``similarly ordered'' functions and may be used to extend that notion into spaces of functions and interpolation of operators. 
\begin{example} Let $(U,\Sigma,\mu)$ be a $\sigma$-finite measure space and fix an integrable $g\in L^+(\Sigma)$. Set $\mathcal A=\{\emptyset\}\cup\{\{u\in U:g(u)> s\}:s>0\}$. Then $u\le_{\mathcal A}v$ means $g(u)\ge g(v)$ and $f\in L^\downarrow(\mathcal A)$ means that $f$ is nonnegative and for all $u,v$, $f(u)\ge f(v)$ if and only if $g(u)\ge g(v)$. Since $g$ is automatically core decreasing, we always have the $D$-type H\"older's inequality mentioned in the introduction. 
\end{example}

\section{A Tailored Measure on the Half Line}\label{tailor}

Let $(U,\Sigma,\mu)$ be a $\sigma$-finite measure space with a full ordered core $\mathcal A$ and let $\mathcal M$ be the enriched core of Theorem \ref{M}. We will use the results of Section \ref{Core Maps} to establish a two-way relationship between functions on $U$ and functions on $[0,\oo)$. This will enable us to take advantage of the familiar order and known results for decreasing functions in that setting. 

Let $\mathcal B=\{\emptyset\}\cup\{[0,x]:x\ge 0\}$. Then $\sigma(\mathcal B)$ is the Borel $\sigma$-algebra. We will construct a measure $\lambda$ on $[0,\oo)$ such that $\lambda([0,x])<\oo$ for each $x>0$ and $\mathcal B$ is a full ordered core on $([0,\oo),\sigma(\mathcal B),\lambda)$. Clearly, $\le_{\mathcal B}$ is the usual order on $[0,\oo)$. 

Let $\Gamma$ be the closure in $[0,\oo)$ of $\mu(\mathcal A)$. Then $\Gamma=\mu(\mathcal M)$ by Theorem \ref{M}. Note that $0\in\Gamma$. Let
\[
a(x)=\sup([0,x]\cap\Gamma)\quad\mbox{and}\quad
b(x)=\inf([x,\oo)\cap\Gamma),
\]
where the infimum of the empty set is taken to be $\oo$. Evidently, $a$ and $b$ are nondecreasing on $[0,\oo)$, $a(x)\le x\le b(x)$, and $a(x)=x=b(x)$ when $x\in\Gamma$. Also, if $x\notin\Gamma$ then $(a(x),b(x))$ is the connected component of the complement of $\Gamma$ that contains $x$. In particular, $b(x)=\oo$ if and only if $x>\sup \Gamma$. Let $\lambda$ denote the Lebesgue-Stieltjes measure associated to the nondecreasing function $b^{-1}$, the generalized inverse of $b$. If $\varphi\in L^+(\mathcal B)$, then
\begin{equation}\label{lam}
\int_{[0,\oo)}\varphi\,d\lambda=\int_0^{\sup\Gamma}\varphi(b(x))\,dx=\int_\Gamma \varphi(x)\,dx+\sum(b-a)\varphi(b),
\end{equation}
where the sum is taken over all bounded components $(a,b)$ of the complement of $\Gamma$.

\begin{lemma}\label{lambda}  The measure $\lambda$ is $\sigma$-finite and supported on $\Gamma$. 
If $x\ge0$, then $\lambda([0,x])=a(x)\in\Gamma$. In particular, $\lambda([0,x])=x$ if and only if $x\in\Gamma$ and $\mathcal B$ is a full ordered core on $([0,\oo),\sigma(\mathcal B),\lambda)$.
\end{lemma}
\begin{proof} For each $x\ge0$ and each $z\le\sup\Gamma$, $a(x),b(z)\in\Gamma$ so $b(z)\le x$ if and only if $b(z)\le a(x)$ if and only if $z\le a(x)$. Therefore, $\chi_{[0,x]}(b(z))=\chi_{[0,a(x)]}(z)$ and we get
\[
\lambda([0,x])=\int_{[0,\oo)}\chi_{[0,x]}\,d\lambda
=\int_0^{\sup\Gamma}\chi_{[0,x]}\circ b
=\int_0^{\sup\Gamma}\chi_{[0,a(x)]}=a(x),
\]
which equals $x$ if and only if $x\in\Gamma$. It follows that $\lambda$ is $\sigma$-finite and $\mathcal B$ is a full ordered core. The third expression in \eqref{lam} shows that $\lambda$ is supported on $\Gamma$.
\end{proof}

Now we apply Theorem \ref{coremap} twice to make connections in both direction between $\mu$-measurable functions on $U$ and $\lambda$-measurable functions on $[0,\oo)$. 

\begin{proposition}\label{R} Let $V(\Sigma)$ be the vector space of all $f\in L(\Sigma)$ such that $f$ is $\mu$-integrable on $M$ for each $M\in \mathcal M$. Then there is a map $R$ defined on $L^+(\Sigma)\cup V(\Sigma)$ such that for all $f\in L^+(\Sigma)\cup V(\Sigma)$:
\begin{enumerate}[label=(\alph*)]
\item\label{Ra} $Rf\in L(\mathcal B)$;
\item\label{Rb} if $f\in L^+(\Sigma)$ then $Rf\in L^+(\mathcal B)$;
\item\label{Rc} if $f_n\in L^+(\Sigma)$ for each $n$ and $f_n\uparrow f$ $\mu$-a.e., then $Rf_n\uparrow Rf$ $\lambda$-a.e.;
\item\label{Rd} $R$ is linear on $V(\Sigma)$ and it is additive and positive homogeneous on $L^+(\Sigma)$;
\item\label{Re} $|Rf|\le R|f|$ $\lambda$-a.e.;
\item\label{Rf} if $x\ge0$, $M\in\mathcal M$, and $\mu(M)=\lambda([0,x])$, then
\[
\int_{[0,x]}Rf\,d\lambda=\int_Mf\,d\mu;
\]
\item\label{Rg} if $f\in L^1_\mu$, then $R f\in L^1_\lambda$ and $\|R f\|_{L^1_\lambda}\le\|f\|_{L^1_\mu}$; 
\item\label{Rh} if $f\in L^\infty_\mu$, then $R f\in L^\oo_\lambda$ and $\|R f\|_{L^\infty_\lambda}\le \|f\|_{L^\infty_\mu}$;
\item\label{Ri} if $M\in\mathcal M$, then $R\chi_M=\chi_{[0,\mu(M)]}$ $\lambda$-a.e.;
\item\label{Rj} if $f\in L^+(\Sigma)$ and $g\in L^+(\mathcal A)$, then $R(fg)=RfRg$ $\lambda$-a.e. 
\end{enumerate}
\end{proposition}
\begin{proof} For each $x\in\Gamma$, choose $M_x\in\mathcal M$ such that $\mu(M_x)=x$.
In Theorem \ref{coremap}, take $(P,\mathcal P,\rho)$ to be $([0,\oo),\sigma(\mathcal B),\lambda)$ with the core $\mathcal B$ and take $(T,\mathcal T,\tau)$ to be $(U,\Sigma,\mu)$. The map $r:\mathcal B\to\Sigma$ is defined by $r(\emptyset)=\emptyset$ and $r([0,x])=M_{a(x)}$. Since $\mu(r(\emptyset))=0$ and $\mu(r([0,x]))=a(x)=\lambda([0,x])<\oo$, the map $r$ satisfies \eqref{bdd} with equality and with $c=1$. 

Parts \ref{Ra}--\ref{Rj} follow directly from the corresponding conclusions of Theorem \ref{coremap}. Only \ref{Rf} and \ref{Ri} require comment. A direct translation of Theorem \ref{coremap}\ref{f}, taking $A=\emptyset$, would be \ref{Rf} above, but with $M_{a(x)}$ in place of $M$. However, $\mu(M_{a(x)})=a(x)=\mu(M)$ so the total ordering of $\mathcal M$ implies that $M$ and $M_{a(x)}$ differ by a set of $\mu$-measure zero. The same issue arises in \ref{Ri}, where a direct translation of Theorem \ref{coremap}\ref{i}, taking $A=[0,\mu(M)]$, gives $R\chi_{M_{a(\mu(M))}}=\chi_{[0,\mu(M)]}$ $\lambda$-a.e. We may replace $M_{a(\mu(M))}$ by $M$ since the two differ by a set of $\mu$-measure zero.
\end{proof}

\begin{proposition}\label{Q}  Let $V(\mathcal B)$ be the vector space of all $\varphi\in L(\mathcal B)$ such that $\varphi$ is $\lambda$-integrable on $[0,x]$ for each $x\ge0$. Then there is a map $Q$ defined on $L^+(\mathcal B)\cup V(\mathcal B)$ such that for all $\varphi\in L^+(\mathcal B)\cup V(\mathcal B)$:
\begin{enumerate}[label=(\alph*)]
\item\label{Qa} $Q\varphi\in L(\mathcal A)$;
\item\label{Qb} if $\varphi\in L^+(\mathcal B)$ then $Q\varphi\in L^+(\mathcal A)$;
\item\label{Qc} if $\varphi_n\in L^+(\mathcal B)$ for each $n$ and $\varphi_n\uparrow \varphi$ $\lambda$-a.e., then $Q\varphi_n\uparrow Q\varphi$ $\mu$-a.e.;
\item\label{Qd} $Q$ is linear on $V(\mathcal B)$ and it is additive and positive homogeneous on $L^+(\mathcal B)$;
\item\label{Qe} $|Q\varphi|\le Q|\varphi|$ $\mu$-a.e.;
\item\label{Qf} for all $M\in\mathcal M$,
\[
\int_MQ\varphi\,d\mu=\int_{[0,\mu(M)]}\varphi\,d\lambda;
\]
\item\label{Qg} if $\varphi\in L^1_\lambda$, then $Q\varphi\in L^1_\mu$ and $\|Q\varphi\|_{L^1_\mu}\le\|\varphi\|_{L^1_\lambda}$; 
\item\label{Qh} if $\varphi\in L^\infty_\lambda$, then $Q\varphi\in L^\oo_\mu$ and $\|Q\varphi\|_{L^\infty_\mu}\le \|\varphi\|_{L^\infty_\lambda}$;
\item\label{Qi} if $M\in\mathcal M$, then $Q\chi_{[0,\mu(M)]}=\chi_M$ $\mu$-a.e.;
\item\label{Qj} if $\varphi\in L^+(\mathcal B)$ and $\psi\in L^+(\mathcal B)$, then $Q(\varphi\psi)=Q\varphi Q\psi$ $\mu$-a.e. 
\end{enumerate}
\end{proposition}
\begin{proof}
In Theorem \ref{coremap}, take $(P,\mathcal P,\rho)$ to be $(U,\Sigma,\mu)$ with core $\mathcal M$ and take $(T,\mathcal T,\tau)$ to be $([0,\oo),\sigma(\mathcal B),\lambda)$. Replace $r$ by the map $q:\mathcal M\to\sigma(\mathcal B)$, defined by $q(\emptyset)=\emptyset$ and $q(M)=[0,\mu(M)]$ otherwise. Since $\lambda(q(\emptyset))=0$ and $\lambda(q(M))=a(\mu(M))=\mu(M)$ when $\emptyset\neq M\in\mathcal M$, the map $q$ satisfies \eqref{bdd} with equality and with $c=1$. 

Parts \ref{Qa}--\ref{Qj} follow directly from the corresponding conclusions of Theorem \ref{coremap}. Only \ref{Qj} requires comment. Translating directly from Theorem \ref{coremap}\ref{j} we get the statement of \ref{Qj} but only under the condition that $\psi$ be measurable in the $\sigma$-algebra generated by $\{[0,x]:x\in\Gamma\}$. This means that $\psi$ must be constant on every component of the complement of $\Gamma$. But $\lambda$ is supported on $\Gamma$ so every $\psi\in L(\mathcal B)$ is equal $\lambda$-a.e. to one that is constant on every component of the complement of $\Gamma$. 
\end{proof}

The next result explores the close connections between the maps $R$ and $Q$.
\begin{theorem}\label{QR} Let $V(\Sigma)$, $V(\mathcal B)$, $R$ and $Q$ be as in Propositions \ref{R} and \ref{Q}. Then:
\begin{enumerate}[label=(\alph*)]
\item\label{QRa} If $\varphi\in L^+(\mathcal B)\cup V(\mathcal B)$, then $RQ\varphi=\varphi$ $\lambda$-a.e.;
\item\label{QRb} If $f\in L^+(\mathcal A)\cup(L(\mathcal A)\cap V(\Sigma))$, then $QRf=f$ $\mu$-a.e.;
\item\label{QRc} If $f\in L^+(\Sigma)$, $\varphi\in L^+(\mathcal B)$ and $M\in\mathcal M$, then 
\[
\int_Mf(Q\varphi)\,d\mu=\int_{[0,\mu(M)]}(Rf)\varphi\,d\lambda\quad\text{and}\quad
\int_Uf(Q\varphi)\,d\mu=\int_{[0,\oo)}(Rf)\varphi\,d\lambda;
\]
\item\label{QRd} If $\varphi\in L^\downarrow(\mathcal B)$, then $Q\varphi\in L^\downarrow(\mathcal A)$ and $\varphi^*=(Q\varphi)^*$;
\item\label{QRe} If $f\in L^\downarrow(\mathcal A)$, then $Rf\in L^\downarrow(\mathcal B)$ and $f^*=(Rf)^*$;
\item\label{QRf} If $g\in L^\downarrow(\mathcal A)$, then
\begin{multline*}
\bigg\{h\in L^\downarrow(\mathcal A):\int_Mh\,d\mu\le\int_Mg\,d\mu\text{ for all }M\in\mathcal M\bigg\}
\\=\bigg\{Q\psi:\psi\in L^\downarrow(\mathcal B), \int_{[0,x]}\psi\,d\lambda\le\int_{[0,x]}Rg\,d\lambda\text{ for all }x\ge0\bigg\}.
\end{multline*}
\end{enumerate}
\end{theorem}
\begin{proof} Recall that if $x\in\Gamma$, $M_x$ is a element of $\mathcal M$ of measure $x$. As in the proofs of Propositions \ref{R} and \ref{Q}, let $r(\emptyset)=\emptyset$, $r([0,x])=M_{a(x)}$, $q(\emptyset)=\emptyset$ and $q(M)=[0,\mu(M)]$ otherwise.

In view of Propositions \ref{R}\ref{Rc} and \ref{Q}\ref{Qc}, it suffices to prove \ref{QRa} for $\varphi\in V(\mathcal B)$. In that case, we may apply Propositions \ref{R}\ref{Rf} and \ref{Q}\ref{Qf}, to get, for all $x\ge0$,
\[
\int_{[0,x]}RQ\varphi\,d\lambda=\int_{M_{a(x)}}Q\varphi\,d\mu=\int_{[0,a(x)]}\varphi\,d\lambda=\int_{[0,x]}\varphi\,d\lambda,
\]
where the last equation relies on Lemma \ref{lambda}, which implies that $\lambda((a(x),x])=0$ when $a(x)<x$. Since $\varphi\in V(\mathcal B)$, each of the above integrals converges. Since the sets $[0,x]$, for $x\ge0$, generate $\sigma(\mathcal B)$, $RQ\varphi=\varphi$ $\lambda$-a.e.

The argument for \ref{QRb} is somewhat similar. In view of Propositions \ref{R}\ref{Rc} and \ref{Q}\ref{Qc}, it suffices to prove \ref{QRb} for $f\in L(\mathcal A)\cap V(\Sigma)$. In that case, since $r([0,\mu(M)])=M_{a(\mu(M))}=M_{\mu(M)}$, by the definition of $a$, we may apply Propositions \ref{R}\ref{Rf} and \ref{Q}\ref{Qf}, to get, for all $M\in\mathcal M$,
\begin{equation}\label{QRflip}
\int_MQRf\,d\mu=\int_{[0,\mu(M)]}Rf\,d\lambda=\int_{M_{\mu(M)}}f\,d\mu=\int_Mf\,d\mu,
\end{equation}
where the last equation relies on $\mu(M_{\mu(M)})=\mu(M)$ and the total order on $\mathcal M$, which together imply that $M_{\mu(M)}$ and $M$ differ by a set of $\mu$-measure zero. Since $f\in V(\Sigma)$, each of the above integrals converges. Both $f$ and $QRf$ are $\sigma(\mathcal A)$-measurable and the set of $M\in\mathcal M$ generate $\sigma(\mathcal A)$ so $QRf=f$ $\mu$-a.e.

To prove \ref{QRc}, let $f\in L^+(\mathcal A)$, $\varphi\in L^+(\mathcal B)$ and $M\in\mathcal M$. As we have seen, $M_{\mu(M)}$ and $M$ differ by a set of $\mu$-measure zero. Also, Proposition \ref{Q}\ref{Qb} shows that $Q\varphi\in L^+(\mathcal A)$ so, applying \ref{Rf} and \ref{Rj} of Proposition \ref{R} followed by \ref{QRa} above, we get
\[
\int_MfQ\varphi\,d\mu
=\int_{[0,\mu(M)]}R(fQ\varphi)\,d\lambda
=\int_{[0,\mu(M)]}(Rf)RQ\varphi\,d\lambda
=\int_{[0,\mu(M)]}(Rf)\varphi\,d\lambda.
\]
To get the second statement of \ref{QRc}, recall that $\mathcal A$ is a full ordered core so $U$ is the union of $A_1\subseteq A_2\subseteq\dots$ with $A_n\in\mathcal A\subseteq \mathcal M$. For such a sequence $[0,\mu(A_n)]$ increases to $[0,\sup\Gamma]$ or to $[0,\oo)$. Since $\lambda$ is supported on $\Gamma$, the monotone convergence theorem implies
\[
\int_UfQ\varphi\,d\mu=\int_{[0,\oo)}(Rf)\varphi\,d\lambda.
\]

Proposition 2.1.7 of \cite{BS} shows that if $0\le\varphi_n\uparrow\varphi$ $\lambda$-a.e., then $0\le\varphi_n^*\uparrow\varphi^*$. Every function in $L^+(\mathcal B)$ is nonnegative and nonincreasing on $[0,\oo)$ so it can be expressed as the limit of functions of the form $\varphi=\sum_{k=1}^K\alpha_k\chi_{[0,x_k]}$ where $\alpha_k>0$ and $x_k\ge0$ for $k=1,\dots,K$. Therefore, it suffices to prove \ref{QRd} for this simple function $\varphi$. Without loss of generality, assume $x_1\ge x_2\ge\dots\ge x_K\ge x_{K+1}=0$. Since $\lambda((a(x),x])=0$ when $a(x)<x$, $\chi_{[0,x_k]}=\chi_{[0,a(x_k)]}=\chi_{[0,\mu(M_{a(x_k)})]}$ $\lambda$-a.e. for each $k=1,\dots,K$. Since $Q$ is additive, Proposition \ref{Q}\ref{Qi} shows that 
\[
Q\varphi=\sum_{k=1}^K\alpha_k\chi_{M_{a(x_k)}}
\]
$\lambda$-a.e., which is in $L^\downarrow(\mathcal A)$ by Lemma \ref{approx}. Also, the only nonzero values taken by $\varphi$ and $Q\varphi$ are $\sum_{k=1}^i\alpha_k$, for $i=1,\dots,K$ with $\varphi$ taking that value on the set $(x_{i+1},x_i]$ and $Q\varphi$ taking that value on $M_{a(x_i)}\setminus M_{a(x_{i+1})}$. Lemma \ref{lambda} shows that $\mu(M_{a(x_i)})=a(x_i)=\lambda([0,x_i])$ for each $i$, so we have $f^*=(Rf)^*$. 

In view of  Lemma \ref{approx} and \cite{BS}*{Proposition 2.1.7}, it suffices to prove \ref{QRe} for functions of the form $f=\sum_{k=1}^K\alpha_k\chi_{M_k}$ where $\alpha_k>0$ and $M_k\in\mathcal M$ for $k=1,\dots,K$. Without loss of generality, assume $M_1\supseteq M_2\supseteq\dots\supseteq M_K\supseteq M_{K+1}=\emptyset$. Since $R$ is additive, Proposition \ref{R}\ref{Ri} shows that $Rf=\sum_{k=1}^K\alpha_k\chi_{[0,\mu(M_k)]}$ $\lambda$-a.e., which is nonincreasing on $[0,\oo)$ and hence in $L^\downarrow(\mathcal B)$. Also, the only nonzero values taken by $f$ and $Rf$ are $\sum_{k=1}^i\alpha_k$, for $i=1,\dots,K$ with $f$ taking that value on the set $M_i\setminus M_{i+1}$ and $Rf$ taking that value on $[0,\mu(M_i)]\setminus[0,\mu(M_{i+1})]$. Since $\mu(M_i)=\lambda([0,\mu(M_i)])$ for each $i$, it follows that $f^*=(Rf)^*$. 

Fix $g\in L^\downarrow(\mathcal A)$. To show that the two sets in \ref{QRf} are equal, we show each contains the other. Suppose $h\in L^\downarrow(\mathcal A)$ satisfies $\int_Mh\,d\mu\le\int_Mg\,d\mu$ for all $M\in\mathcal M$. Take $\psi=Rh$ and use \ref{QRe} and \ref{QRb} to get $\psi\in L^\downarrow(\mathcal B)$ and $h=QRh=Q\psi$. For $x\ge0$, choose $M\in\mathcal M$ such that $\mu(M)=a(x)$ and apply Proposition \ref{R}\ref{Rf} twice to get
\[
\int_{[0,x]}\psi\,d\lambda=\int_{[0,x]}Rh\,d\lambda=\int_Mh\,d\mu\le\int_Mg\,d\mu=\int_{[0,x]}Rg\,d\lambda.
\]
This proves ``$\subseteq$''.

For the reverse inclusion, suppose $\psi\in L^\downarrow(\mathcal B)$ satisfies $\int_{[0,x]}\psi\,d\lambda\le\int_{[0,x]}Rg\,d\lambda$ for all $x\ge0$ Take $h=Q\psi$ and use \ref{QRd} to get $h\in L^\downarrow(\mathcal A)$. For $M\in\mathcal M$, apply Proposition \ref{Q}\ref{Qf} and Proposition \ref{R}\ref{Rf} to get
\[
\int_Mh\,d\mu=\int_MQ\psi\,d\mu=\int_{[0,\mu(M)]}\psi\,d\lambda\le\int_{[0,\mu(M)]}Rg\,d\lambda=\int_Mg\,d\mu.
\]
This proves ``$\supseteq$'' and completes the proof of \ref{QRf}.
\end{proof}
\section{Spaces Defined by Core Decreasing Functions}\label{decreasing}

To use the properties of decreasing functions as an effective tool of functional analysis, we need to relate them to Banach spaces and operators on Banach spaces. This was done for decreasing functions on $\mathbb R$ using the level function and least decreasing majorant constructions. In this section we define the down space of a Banach function space using a restricted duality approach and give analogues of these two constructions for core decreasing functions on a measure space.

Let $(U,\Sigma,\mu)$ be a $\sigma$-finite measure space with a full ordered core $\mathcal A$, let $M$ be the enriched core of Theorem \ref{M}, let $\lambda$ be the measure of Section \ref{tailor}, and let $R$ and $Q$ be the maps introduced in Propositions \ref{R} and \ref{Q}. We assume that $\mathcal A$ is full to ensure that the spaces we define below will be Banach spaces. Since core decreasing functions vanish on $U\setminus\cup\mathcal A$, without fullness the spaces would be only semi-normed.

Let $X$ be a Banach function space over $U$ and assume that both $X$ and $X'$ contain all characteristic functions of sets of finite measure. This assumption ensures that our spaces have plenty of core decreasing functions. It also ensures that $X'$ is a normed, not just seminormed. Observe that every nontrivial u.r.i. space satisfies the assumption. 
\begin{definition}\label{down} For a $\mu$-measurable function $f$, let
\[
\|f\|_{X\!\downarrow}=\sup\bigg\{\int_U |f|g\,d\mu:g\in X'\cap L^\downarrow(\mathcal A), \|g\|_{X'}\le1\bigg\}
\]
and $X\!\!\downarrow=\{f\in L_\mu:\|f\|_{X\!\downarrow}<\oo\}$. We call $X\!\!\downarrow$ the \emph{down space} of $X$.
\end{definition}
\begin{definition}\label{deflev} Let $f\in L(\Sigma)$. We say $f^o\in L^\downarrow(\mathcal A)$ is a \emph{level function} of $f$ if for all $g\in L^\downarrow(\mathcal A)$, 
\[
\int_U f^og\,d\mu=\sup\bigg\{\int_U|f|h\,d\mu:h\in L^\downarrow(\mathcal A), \int_Ah\,d\mu\le\int_Ag\,d\mu\text{ for all } A\in\mathcal A\bigg\}.
\]

Let $X^o$ be the collection of $f\in L(\Sigma)$ such that $f$ has a level function in $X$ and set $\|f\|_{X^o} =\|f\|_X$ for all $f\in X^o$.
\end{definition}
\begin{definition} Let $g\in L(\Sigma)$. We call $h$ a \emph{core decreasing majorant} of $g$ if $h\in L^\downarrow(\mathcal A)$ and $|g|\le h$ $\mu$-a.e. We call $\widetilde g$ a \emph{least core decreasing majorant} of $g$ if it is a core decreasing majorant and for every core decreasing majorant $h$, $\widetilde g\le h$ $\mu$-a.e.

Let $\widetilde X$ be the collection of $g\in L(\Sigma)$ such that $g$ has a least core decreasing majorant in $X$ and set $\|g\|_{\widetilde X}=\|\widetilde g\|_X$ for all $g\in \widetilde X$.
\end{definition}
It is evident that a least core decreasing majorant is unique up to equality $\mu$-a.e., provided it exists. We will prove existence in Lemma \ref{lcdm}. The existence and uniqueness of level functions is discussed in Lemma \ref{level}.

A level function, above, is defined in terms of the original ordered core $\mathcal A$, but, using Theorem \ref{M} and Lemma \ref{3fae}\ref{sC}, it is easy to show that it would not change if defined in terms of the enriched core $M$. That is, for all $g\in L^\downarrow(\mathcal A)$,
\begin{equation}\label{AorM}
\int_U f^og\,d\mu=\sup\bigg\{\int_U|f|h\,d\mu:h\in L^\downarrow(\mathcal A), \int_Mh\,d\mu\le\int_Mg\,d\mu\text{ for all } M\in\mathcal M\bigg\}.
\end{equation}

Computation of the norms of ``endpoint'' down spaces will illustrate Definition \ref{down}. They will be key spaces in the interpolation theory of Section \ref{interp}. 
\begin{theorem}\label{Loo} Let $f\in L(\Sigma)$ and $\varphi\in L(\mathcal B)$. Then
\begin{enumerate}[label=(\alph*)]
\item\label{Looa}  $L_\mu^1\!\!\downarrow=L_\mu^1$, with identical norms;
\item\label{Loob}  $\|f\|_{L_\mu^\oo\!\downarrow}=\sup_{A\in\mathcal A}\frac1{\mu(A)}\int_A|f|\,d\mu=\sup_{M\in\mathcal M}\frac1{\mu(M)}\int_M|f|\,d\mu$;
\item\label{Looc} if $f\in L_\mu^\oo\!\!\downarrow$ then $Rf\in L_\lambda^\oo\!\!\downarrow$ and $\|Rf\|_{L_\lambda^\oo\!\downarrow}\le\|f\|_{L_\mu^\oo\!\downarrow}$;
\item\label{Lood} if $\varphi\in L_\lambda^\oo\!\!\downarrow$ then $Q\varphi\in L_\mu^\oo\!\!\downarrow$ and $\|Q\varphi\|_{L_\mu^\oo\!\downarrow}\le\|\varphi\|_{L_\lambda^\oo\!\downarrow}$.
\end{enumerate}
\end{theorem}
\begin{proof} First recall that $(L^1_\mu)'=L^\oo_\mu$ and $(L^\oo_\mu)'=L^1_\mu$. 

If $\|g\|_{L_\mu^\oo}\le 1$, then $\int_U|f|g\,d\mu\le\|f\|_{L_\mu^1}$ so $\|f\|_{L_\mu^1\!\downarrow}\le\|f\|_{L_\mu^1}$. On the other hand, taking $g$ to be the constant function $1$, we have $g\in L^\downarrow(\mathcal A)$ and $\|g\|_{L_\mu^\oo}\le 1$, so $\|f\|_{L_\mu^1}=\int_U|f|g\,d\mu\le\|f\|_{L_\mu^1\!\downarrow}$. This proves \ref{Looa}. 
 
Theorem \ref{M} and Lemma \ref{3fae}\ref{sC} show that the values of the two suprema in \ref{Loob} coincide. Let $\nu$ be their common value. Fix $M\in \mathcal M$ with $\mu(M)>0$ and set $g=(1/\mu(M))\chi_M$. Then $\|g\|_{L_\mu^1}\le 1$, $g\in L^\downarrow(\mathcal A)$, and
\[
\frac1{\mu(M)}\int_M|f|\,d\mu=\int_U|f|g\,d\mu\le \|f\|_{L_\mu^\oo\!\downarrow}.
\]
Take the supremum over all such $M$ to get $\nu\le\|f\|_{L_\mu^\oo\!\downarrow}$. On the other hand, $\int_M|f|\,d\mu\le\int_M\nu\,d\mu$ for all $M\in\mathcal M$. Corollary \ref{M2g} shows that for all $g\in L^\downarrow(\mathcal A)$, 
\[
\int_U|f|g\,d\mu\le\int_U\nu g\,d\mu=\nu\|g\|_{L^1_\mu}.
\]
Take the supremum over all such $g$ with $\|g\|_{L^1_\mu}\le1$ to get $\|f\|_{L_\mu^\oo\!\downarrow}\le\nu$. This proves \ref{Loob}.

If $f\in L^\oo_\mu\!\!\downarrow$, fix $x\ge0$ with $\lambda([0,x])>0$, and, by Lemma \ref{lambda}, choose $M\in\mathcal M$ such that $\mu(M)=\lambda([0,x])$. By Proposition \ref{R}, parts \ref{Re} and \ref{Rf},
\[
\frac1{\lambda([0,x])}\int_{[0,x]}|Rf|\,d\lambda
\le\frac1{\lambda([0,x])}\int_{[0,x]}R|f|\,d\lambda
=\frac1{\mu(M)}\int_M|f|\,d\mu\le\|f\|_{L^\oo_\mu\!\downarrow}.
\]
Take the supremum over such $x$ and use the expression for $\|\varphi\|_{L^\oo_\lambda\!\downarrow}$ given in Section \ref{old} to get \ref{Looc}.

If $\varphi\in L^\oo_\lambda\!\!\downarrow$, fix $M\in\mathcal M$ with $\mu(M)>0$. By Lemma \ref{lambda}, $\lambda([0,\mu(M)])=\mu(M)$ so, by Proposition \ref{Q}, parts \ref{Qe} and \ref{Qf},
\[
\frac1{\mu(M)}\int_M|Q\varphi|\,d\mu
\le\frac1{\mu(M)}\int_MQ|\varphi|\,d\mu
=\frac1{\lambda([0,\mu(M)])}\int_{[0,\mu(M)]}|\varphi|\,d\lambda\le\|\varphi\|_{L^\oo_\lambda\!\downarrow}.
\]
Take the supremum over such $M$ to get \ref{Lood}.
\end{proof}

Next we show that each $f\in L(\Sigma)$ has a level function, which is unique if $f$ is bounded. If $f$ has a level function which takes the value $\oo$ on a set of positive $\mu$-measure, then one easily verifies that the constant function $\oo$ is also a level function of $f$, so we can't expect that every $f$ has a unique level function. The next lemma shows there is a consistent choice of level function (effectively choosing the smallest one) that justifies referring to ``the'' level function of $f$ even when $f^o$ takes infinite values. 

\begin{lemma}\label{level}
There is a unique map $f\mapsto f^o$ from $L(\Sigma)$ to $L^\downarrow(\mathcal A)$ such that $f^o$ is a level function of $f$ and if $0\le f_n\uparrow f$ $\mu$-a.e., then $f_n^o\uparrow f^o$ $\mu$-a.e. 
\end{lemma}
\begin{proof} Let $f\in L^+(\Sigma)$, let $(Rf)^o$ be the level function from Section \ref{old}, and set $f^o=Q((Rf)^o)$. Since $(Rf)^o\in L^\downarrow(\mathcal B)$, Theorem \ref{QR}\ref{QRd} shows $f^o\in L^\downarrow(\mathcal A)$. Now fix $g\in L^\downarrow(\mathcal A)$. Theorem \ref{QR}\ref{QRe} shows $Rg\in L^\downarrow(\mathcal B)$ and Theorem \ref{QR}\ref{QRc} together with results of Section \ref{old} show
\begin{align*}
\int_Uf^og\,d\mu&=\int_{[0,\oo)}(Rf)^oRg\,d\lambda\\
&=\sup\bigg\{\int_{[0,\oo)}(Rf)\psi\,d\lambda:\psi\in L^\downarrow(\mathcal B),\int_{[0,x]}\psi\,d\lambda\le \int_{[0,x]}Rg\,d\lambda\text{ for }x\ge0\bigg\}.
\end{align*}
But $\int_{[0,\oo)}(Rf)\psi\,d\lambda=\int_UfQ\psi\,d\mu$ by Theorem \ref{QR}\ref{QRc}, so in view of Theorem \ref{QR}\ref{QRf} we have
\[
\int_U f^og\,d\mu=\sup\bigg\{\int_Ufh\,d\mu:h\in L^\downarrow(\mathcal A), \int_Mh\,d\mu\le\int_Mg\,d\mu\text{ for all } M\in\mathcal M\bigg\}.
\]
By \eqref{AorM}, $f^o$ is a level function of $f$. If $f\in L(\Sigma)$, set $f^o=(|f|)^o$. It is immediate that $f^o$ is a level function of $f$. 

If $0\le f_n\uparrow f$ $\mu$-a.e., then Proposition \ref{R}, parts \ref{Rb} and \ref{Rc}, imply that $0\le Rf_n\uparrow Rf$ $\lambda$-a.e. and the results of Section \ref{old} show that $0\le (Rf_n)^o\uparrow (Rf)^o$ $\lambda$-a.e. From Proposition \ref{Q}\ref{Qc}, we get $(f_n)^o\uparrow f^o$ $\mu$-a.e.

Suppose $f\mapsto f^\bullet$ is another map from $L(\Sigma)$ to $L^\downarrow(\mathcal A)$ such that $f^\bullet$ is a level function of $f$ and if $0\le f_n\uparrow f$ $\mu$-a.e., then $f_n^\bullet\uparrow f^\bullet$ $\mu$-a.e. Then, for each $f\in L(\Sigma)\cap   L^\oo_\mu$, $f^\bullet$ and $f^o$ are non-negative, $\sigma(\mathcal M)$-measurable and, taking $g=\chi_M$ in Definition \ref{deflev}, satisfy $\int_Mf^\bullet\,d\mu=\int_Mf^o\,d\mu\le\|f\|_{L^\oo_\mu}\mu(M)<\oo$ for all $M\in\mathcal M$. It follows that $f^\bullet=f^o$ $\mu$-a.e.

For every $f\in L(\Sigma)$, $|f|$ can be expressed as the limit of an increasing sequence of $f_n\in L^+(\Sigma)\cap L^\oo_\mu$. Since $f_n^\bullet\uparrow f^\bullet$, $f_n^o\uparrow f^o$ $\mu$-a.e. and $f_n^\bullet=f_n^o$ $\mu$-a.e for all $n$, we have $f^\bullet=f^o$ $\mu$-a.e. This proves uniqueness. 
\end{proof}

As Example \ref{ordinal} shows, if $g\in L(\Sigma)$, the natural analogue of the definition of $\widetilde \psi$ given in Section \ref{old} may fail to give a least core decreasing majorant of $g$. The next lemma takes a different approach.
\begin{lemma}\label{lcdm} Every $g\in L(\Sigma)$ has a least core decreasing majorant $\widetilde g$, which is unique up to equality $\mu$-a.e. If $g_n\in L^+(\Sigma)$ and $g_n\uparrow g$ $\mu$-a.e., then $\widetilde{g_n}\uparrow \widetilde g$ $\mu$-a.e. 
\end{lemma}
\begin{proof} Let $g\in L(\Sigma)$. Choose an increasing sequence of core sets $A_n$ such that $U=\cup_{n=1}^\oo A_n$. This is possible because $\mathcal A$ is a full ordered core. The functions $g$ and $|g|$ have the same core decreasing majorants so we may assume $g$ is non-negative. First suppose $g$ is bounded above so $g$ has a finite core decreasing majorant. For each $n$, let $\beta_n=\inf\int_{A_n}h\,d\mu<\oo$, where the infimum is taken over all core decreasing majorants $h$ of $g$, and choose a core decreasing majorant $h_n$ of $g$ such that $\int_{A_n}h_n\,d\mu<\beta_n+1/n$. Since the minimum of two core decreasing majorants of $g$ is again a core decreasing majorant of $g$, we may assume that $h_1\ge h_2\ge\dots$. The limit of $h_n$, call it $h$, exists and is a core decreasing majorant of $g$. If $f$ is a core decreasing majorant of $g$, so is $\min(f,h)$. Thus, for each $n$,
\[
\int_{A_n}h\,d\mu\le\int_{A_n}h_n\,d\mu<\beta_n+\frac1n\le\int_{A_n}\min(f,h)\,d\mu+\frac1n<\oo.
\]
Thus
\[
\int_{A_n}h-\min(f,h)\,d\mu\le\frac1n.
\]
The function $h-\min(f,h)$ is nonnegative and, letting $n\to\oo$ we see that its integral on $U$ is zero. Therefore $f\ge h$ $\mu$-a.e. This shows that $h$ is a least core decreasing majorant of $g$. 

If $g$ is not bounded above, let $h_k$ be a least core decreasing majorant of $\min(g,k)$ for $k=1,2,\dots$. Then $h_k$ is an increasing sequence and it is easy to see that its limit is a least core decreasing majorant of $g$.

Uniqueness, up to equality $\mu$-a.e., is clear. Finally, if $0\le g_n\uparrow g$ $\mu$-a.e., it is immediate that $\widetilde{g_n}$ increases with $n$ $\mu$-a.e. Moreover, its limit is a least core decreasing majorant of $g$. Uniqueness shows that the limit is equal to $\widetilde g$ $\mu$-a.e.
\end{proof}

With our two constructions in place, we give some properties of the down space and its associate space $X'$. In Theorem \ref{downlev} we will get a stronger conclusion for $X\!\!\downarrow$ and $X^o$ when $X$ is assumed to be u.r.i.
\begin{theorem}\label{props} The set $X\!\!\downarrow$, equipped with the norm  $\|\cdot\|_{X\!\downarrow}$, is a Banach function space with the Fatou property. If $f\in X''$, then $\|f\|_{X\!\downarrow}\le\|f\|_{X''}$ and, if $f^o\in X''$, then $\|f\|_{X\!\downarrow}\le\|f^o\|_{X''}$. 

The map $g\mapsto\widetilde g$ is sublinear. The set $\widetilde X$, equipped with the norm, $\|\cdot\|_{\widetilde X}$ is a Banach function space. It has the Fatou property if $X$ does. Also, $X\!\!\downarrow=\big(\widetilde{X'}\big)'$, with identical norms, and $X\!\!\downarrow'=\widetilde{X'}$, with identical norms.

Finally, $\widetilde{L^\oo_\mu}=L^\oo_\mu$, with identical norms.
\end{theorem}
\begin{proof} The map $f\mapsto\|f\|_{X\!\downarrow}$ is nonnegative, positive homogeneous, and satisfies the triangle inequality. This ensures that $X\!\!\downarrow$ is a real vector space. The lattice property from Section \ref{uri} and the Fatou property, which implies completeness, follow by standard arguments, the same as those that prove the associate space of a saturated Banach function space is a Banach function space with the Fatou property.

Now suppose that $\|f\|_{X\!\downarrow}=0$. If $M\in\mathcal M$ and $\mu(M)>0$, then $\chi_{M}\in X'$ and $0<\|\chi_M\|_{X'}<\oo$. But $\chi_M\in L^\downarrow(\mathcal A)$ so
\[
0=\|f\|_{X\!\downarrow}\ge\int_U|f|\frac{\chi_M}{\|\chi_M\|_{X'}}\,d\mu=\frac1{\|\chi_M\|_{X'}}\int_M|f|\,d\mu.
\]
Since $\mathcal A$ is a full ordered core, this implies $\int_U|f|\,d\mu=0$ so $f=0$ $\mu$-a.e. This proves that $X\!\!\downarrow$ is a Banach function space with the Fatou property. 

Dropping the condition $g\in L^\downarrow(\mathcal A)$, in the definition of $\|f\||_{X\!\downarrow}$, gives the definition of $\|f\|_{X''}$. Thus $\|f\|_{X\!\downarrow}\le\|f\|_{X''}$.

If $g\in L^\downarrow(\mathcal A)$ and $\|g\|_{X'}\le1$ then, taking $h=g$ in the definition of a level function, $\int_U|f|g\,d\mu\le\int_Uf^og\,d\mu\le\|f^o\|_{X''}$. Taking the supremum over all such $g$, we have $\|f\|_{X\!\downarrow}\le\|f^o\|_{X''}$. 

The map $g\mapsto\widetilde g$ is clearly positive homogeneous and is also subadditive: If $f=g+h$, then $|f|\le|g|+|h|\le\widetilde g+\widetilde h$, which is core decreasing, so  $\widetilde f\le\widetilde g+\widetilde h$. Thus the map is sublinear. This makes it easy to verify that $\|\cdot\|_{\widetilde X}$ is a norm on $\widetilde X$. If $|f|\le g$ then $\widetilde f\le\widetilde g$ so $\widetilde X$ satisfies the lattice property of Section \ref{uri}. To prove $\widetilde X$ is complete, we will prove that it has the Riesz-Fischer property, that is, if $0\le g_n\in\widetilde X$ satisfy $\sum_{n=1}^\oo\|g_n\|_{\widetilde X}<\oo$, then $\|\sum_{n=1}^\oo g_n\|_{\widetilde X}\le\sum_{n=1}^\oo\|g_n\|_{\widetilde X}$. See \cite{Z}*{15.64.2}. For such functions $g_n$, let $G=\sum_{n=1}^\oo g_n$ and observe that $\sum_{n=1}^\oo \widetilde g_n$ is a core decreasing majorant of $G$ and hence majorizes $\widetilde G$. Since $X$ is complete, it has the Riesz-Fischer property, so 
\[
\|G\|_{\widetilde X}=\big\|\widetilde G\big\|_X\le\Big\|\sum_{n=1}^\oo \widetilde g_n\Big\|_X
\le\sum_{n=1}^\oo \|\widetilde g_n\|_X=\sum_{n=1}^\oo \|g_n\|_{\widetilde X}.
\]
Thus, $\widetilde X$ is a Banach function space. Finally, if $X$ has the Fatou property, and $0\le g_n\uparrow g$ $\mu$-a.e., then $\widetilde g_n\uparrow \widetilde g$ $\mu$-a.e. so $\|\widetilde g_n\|_X\uparrow \|\widetilde g\|_X$. Thus $\widetilde X$ also has the Fatou property.

If $g\in L^\downarrow(\mathcal A)$ and $\|g\|_{X'}\le1$, then $\widetilde g=g$ so we have $\|g\|_{\widetilde{X'}}\le1$ and conclude that $\int_U |fg|\,d\mu\le\|f\|_{(\widetilde{X'})'}$. Take the supremum over all such $g$ to get $\|f\|_{X\!\downarrow}\le\|f\|_{(\widetilde{X'})'}$. On the other hand, if $\|g\|_{\widetilde{X'}}\le1$ then $\widetilde g\in L^\downarrow(\mathcal A)$ and $\|\widetilde g\|_{X'}\le1$ so the definition of $X\!\!\downarrow$ gives $\int_U|fg|\,d\mu\le\int_U|f|\widetilde g\,d\mu\le\|f\|_{X\!\downarrow}$. Taking the supremum over all such $g$ gives $\|f\|_{(\widetilde{X'})'}\le\|f\|_{X\!\downarrow}$. This shows that $X\!\!\downarrow=\big(\widetilde{X'}\big)'$, with identical norms.

It follows that $X\!\!\downarrow'=\big(\widetilde{X'}\big)''$, with identical norms. But $X'$ has the Fatou property, and therefore so does $\widetilde{X'}$. Thus $X\!\!\downarrow'=\big(\widetilde{X'}\big)''=\widetilde{X'}$, with identical norms. 

Since $|g|\le\widetilde g$ $\mu$-a.e,  $\|g\|_{L^\oo_\mu}\le\|g\|_{\widetilde{L^\oo_\mu}}$. But the constant function with value $\|g\|_{L^\oo_\mu}$ is trivially core decreasing and majorizes $g$ so $\widetilde g\le\|g\|_{L^\oo_\mu}$. Thus $\|g\|_{\widetilde{L^\oo_\mu}}\le\|g\|_{L^\oo_\mu}$ as well. So $\widetilde{L^\oo_\mu}=L^\oo_\mu$, with identical norms.
\end{proof}
 
\section{Interpolation}\label{interp}

Here we express the $K$-functional for the couple $(L^1_\mu,L^\oo_\mu\!\!\downarrow)$ in terms of the level function and the $K$-functional for the couple $(\widetilde{L^1_\mu},L^\oo_\mu)$ in terms of the least core decreasing majorant. Then we show that both couples are Calder\'on couples  with divisibility constant $1$. Finally, we describe all the interpolation spaces for both couples by relating them to u.r.i. spaces.

Let $(U,\Sigma,\mu)$ be a $\sigma$-finite measure space with a full ordered core $\mathcal A$, let $M$ be the enriched core of Theorem \ref{M}, let $\lambda$ be the measure of Section \ref{tailor}, and let $R$ and $Q$ be the maps introduced in Propositions \ref{R} and \ref{Q}. 

The mapping properties of $Q$ and $R$ show that $Q$ is an admissible contraction from $(L^1_\lambda,L^\oo_\lambda)$ to $(L^1_\mu,L^\oo_\mu)$ and $R$ is an admissible contraction from $(L^1_\mu,L^\oo_\mu)$ to $(L^1_\lambda,L^\oo_\lambda)$. This remains valid for other couples of interest.
\begin{lemma}\label{AM} The maps $Q$ and $R$ have the following mapping properties:
\begin{enumerate}[label=(\alph*)]
\item\label{AMa} $Q$ is an admissible contraction from $(L^1_\lambda, L^\oo_\lambda\!\!\downarrow)$ to $(L^1_\mu, L^\oo_\mu\!\!\downarrow)$;
\item\label{AMb} $R$ is an admissible contraction from $(L^1_\mu, L^\oo_\mu\!\!\downarrow)$ to $(L^1_\lambda, L^\oo_\lambda\!\!\downarrow)$;
\item\label{AMc} $Q$ is an admissible contraction from $(\widetilde{L^1_\lambda},L^\oo_\lambda)$ to $(\widetilde{L^1_\mu}, L^\oo_\mu)$;
\item\label{AMd} $R$ is an admissible contraction from $(\widetilde{L^1_\mu}, L^\oo_\mu)$ to $(\widetilde{L^1_\lambda}, L^\oo_\lambda)$.
\end{enumerate}
\end{lemma}
\begin{proof} Proposition \ref{Q}\ref{Qg} and Theorem \ref{Loo}\ref{Lood} prove \ref{AMa}. Proposition \ref{R}\ref{Rg} and Theorem \ref{Loo}\ref{Looc} prove \ref{AMb}. If $\psi\in \widetilde{L^1_\lambda}$, then $|Q\psi|\le Q|\psi|\le Q\widetilde\psi\in L^\downarrow(\mathcal A)$, by parts \ref{Qb} and \ref{Qd} of Proposition \ref{Q}. It follows that $\widetilde{Q\psi}\le Q\widetilde\psi$. Thus,
\[
\|Q\psi\|_{\widetilde{L^1_\mu}}=\|\widetilde{Q\psi}\|_{L^1_\mu}\le\|Q\widetilde\psi\|_{L^1_\mu}\le\|\widetilde\psi\|_{L^1_\lambda}=\|\psi\|_{\widetilde{L^1_\lambda}}.
\]
This inequality, combined with Proposition \ref{Q}\ref{Qh}, proves \ref{AMc}. If $g\in \widetilde{L^1_\mu}$, then $|Rg|\le R|g|\le R\widetilde g\in L^\downarrow(\mathcal B)$, by parts \ref{Rb} and \ref{Rd} of Proposition \ref{R}. It follows that $\widetilde{Rg}\le R\widetilde g$. Thus,
\[
\|Rg\|_{\widetilde{L^1_\lambda}}=\|\widetilde{Rg}\|_{L^1_\lambda}\le\|R\widetilde g\|_{L^1_\lambda}\le\|\widetilde g\|_{L^1_\mu}=\|g\|_{\widetilde{L^1_\mu}}.
\]
This inequality, combined with Proposition \ref{R}\ref{Rh}, proves \ref{AMd}.
\end{proof}

If $X$ is u.r.i. and has the Fatou property, the next theorem shows that $X^o$ is a Banach function space by showing that it agrees with $X\!\!\downarrow$. It will follow from Theorem \ref{summ} that $X^o$ is a Banach function space for any u.r.i. space $X$.
\begin{theorem}\label{downlev} If $X$ is a u.r.i. space then $X\!\!\downarrow=(X'')^o$, with identical norms. If $X$ also has the Fatou property, then $X\!\!\downarrow=X^o$.
\end{theorem} 
\begin{proof}From Theorem \ref{props} we have $\|f\|_{X\!\downarrow}\le\|f^o\|_{X''}=\|f\|_{(X'')^o}$ for all $f\in (X'')^o$. 

If $f\in X\!\!\downarrow$ we can choose non-negative functions $f_n$, with each one bounded above and supported on some $A_n\in\mathcal A$, that satisfy $f_n\uparrow |f|$ and hence $(f_n)^o\uparrow f^o$ $\mu$-a.e. Since $X''$ and $X\!\!\downarrow$ have the Fatou property, $\|(f_n)^o\|_{X''}$ increases to $\|f^o\|_{X''}$ and $\|f_n\|_{X\!\downarrow}$ increases to $\|f\|_{X\!\downarrow}$. Therefore, when proving $\|f^o\|_{X''}\le\|f\|_{X\!\downarrow}$ we are free to assume that $f$ is nonnegative, bounded above, and supported on some $A\in\mathcal A$.  

Fix such an $f$ and apply Proposition \ref{R}, parts \ref{Rb}, \ref{Rh}, \ref{Ri} and \ref{Rj}, to see that $Rf$ is non-negative, bounded above and supported on $[0,x]$ for some $x\in[0,\oo)$. By \cite{Stm}*{Proposition 1.5}, $(Rf)^o=J_f(Rf)$, where the operator $J_f$ is defined by
\[
J_f\varphi(x)=\begin{cases}\frac1{\lambda(I)}\int_I\varphi\,d\lambda,&x\in I\in\mathcal I;\\\varphi(x),&x\notin\cup_{I\in\mathcal I}I;\end{cases}
\]
for some countable collection $\mathcal I$ of disjoint subintervals of $[0,\oo)$ depending on $f$. By \cite{Stm}*{Proposition 1.4}, we see that if $\varphi,\psi\in L^+(\mathcal B)$, then $\int_{(0,\oo]}\psi J_f\varphi\,d\lambda=\int_{(0,\oo]}\varphi J_f\psi\,d\lambda$ and if $\varphi\in L^\downarrow(\mathcal B)$ then $J_f\varphi\in L^\downarrow(\mathcal B)$. It is clear that $\|J_f\varphi\|_{L^\oo_\lambda}\le\|\varphi\|_{L^\oo_\lambda}$, and the formal self-adjointness just mentioned implies $\|J_f\psi\|_{L^1_\lambda}\le\|\psi\|_{L^1_\lambda}$. Thus $J_f$ is an admissible contraction on $(L^1_\lambda,L^\oo_\lambda)$.

Fix a nonnegative $g$ that is also bounded above and supported on some $A\in\mathcal A$ and let $J_g$ be the corresponding operator for $g$, so $(Rg)^o=J_g(Rg)$. Using the formulas  $f^o=Q((Rf)^o)$ and $g^o=Q((Rg)^o)$ given in the proof of Lemma \ref{level}, we have $f^og^o=(QJ_fRf)(QJ_gRg)=Q((J_fRf)(J_gRg))$, by Proposition \ref{Q}\ref{Qj}. So Definition \ref{deflev}, Theorem  \ref{QR}\ref{QRc} and formal self-adjointness yield
\[
\int_Uf^og\,d\mu\le\int_Uf^og^o\,d\mu
=\int_{[0,\oo)}(J_fRf)(J_gRg)\,d\lambda=\int_UfQJ_fJ_gRg\,d\mu.
\]
Since $X'$ is u.r.i., it is an exact interpolation space for $(L^1_\mu,L^\oo_\mu)$. It follows from the comment before Lemma \ref{AM} that $QJ_fJ_gR$ is an admissible contraction from $(L^1_\mu,L^\oo_\mu)$ to itself so $\|QJ_fJ_gRg\|_{X'}\le\|g\|_{X'}$. But $QJ_fJ_gRg=QJ_f(Rg)^o\in L^\downarrow(\mathcal A)$ and therefore,
\[
\int_Uf^og\,d\mu\le\int_{[0,\oo)}fQJ_fJ_gRg\,d\lambda\le\|f\|_{X\!\downarrow}\|g\|_{X'}.
\]
For any $g\in X'$, with $\|g\|_{X'}\le1$, we can express $|g|$ as the pointwise limit of a increasing sequence of functions $g_n$ of the above form. By the monotone convergence theorem,
\[
\int_Uf^o|g|\,d\mu\le\|f\|_{X\!\downarrow}.
\]
Taking the supremum over all such $g$ proves $\|f^o\|_{X''}=\|f\|_{X\!\downarrow}$.

If $X$ has the Fatou property, then $X''=X$ and the last statement follows.
\end{proof}

Next, to help with the $K$-functional for $(L^1_\mu,L^\oo_\mu\!\!\downarrow)$, we build a family of decompositions for functions in $L^1_\mu+L^\oo_\mu\!\!\downarrow$.
\begin{lemma}\label{decomp} For $0\le f\in L^1_\mu+L^\oo_\mu\!\!\downarrow$ there is a map $D_f:[0,\oo)\to L^\downarrow(\mathcal A)$ such that: $0\le D_f(\gamma)\le 1$ for all $\gamma\ge0$; if $\int_Mg\,d\mu=\int_Mf\,d\mu$ for all $M\in \mathcal M$, then $D_f=D_g$ $\mu$-a.e.; and if $f=f_1+f_\oo$ with $0\le f_1\in L^1_\mu$ and $0\le f_\oo\in L^\oo_\mu\!\!\downarrow$, then 
\[
\|D_f(\gamma)f\|_{L^1_\mu}\le\|f_1\|_{L^1_\mu}\quad\text{and}\quad
\|(1-D_f(\gamma))f\|_{L^\oo_\mu\!\downarrow}\le\|f_\oo\|_{L^\oo_\mu\!\downarrow}
\]
when $\gamma=\|f_1\|_{L^1_\mu}$.
\end{lemma}
\begin{proof} For each map $\Theta:\mathcal M\to[0,\oo)$ and each $x\in\Theta(\mathcal M)$ chose  an $N_x\in\mathcal M$ such that $\Theta(N_x)=x$. This a priori choice of $N_x$ will avoid issues with possible incompatible choices later.

Fix a nonnegative $f\in L^1_\mu+L^\oo_\mu\!\!\downarrow$ and suppose $0\le f_1\in L^1_\mu$ and $0\le f_\oo\in L^\oo_\mu\!\!\downarrow$ such that $f=f_1+f_\oo$. For each $M\in\mathcal M$, 
\begin{equation}\label{overM}
\int_Mf\,d\mu=\int_Mf_1\,d\mu+\int_Mf_\oo\,d\mu
\le\|f_1\|_{L^1_\mu}+\|f_\oo\|_{L^\oo_\mu\!\downarrow}\mu(M)<\oo.
\end{equation}
Set $\Theta(M)=\int_Mf\,d\mu$ and let $N_x$, for $x\in\Theta(\mathcal M)$, be those determined above. By Theorem \ref{M}, $\mathcal M$ is closed under countable unions and intersections, showing that $\Theta(\mathcal M)=\{\int_M|f|\,d\mu:M\in\mathcal M\}$ is a closed subset of $[0,\oo)$ containing $0$. For each $\gamma\ge0$, set 
\[
a_\gamma=\sup[0,\gamma]\cap\Theta(\mathcal M)\quad\text{and}\quad b_\gamma=\inf[\gamma,\oo)\cap\Theta(\mathcal M),
\]
where $\inf\emptyset=\oo$. Then $a_\gamma\in\Theta(\mathcal M)$, $b_\gamma\in\Theta(\mathcal M)\cup\{\oo\}$ and either $a_\gamma=\gamma=b_\gamma$ or $a_\gamma<\gamma<b_\gamma\le\oo$.

If $a_\gamma<\gamma<b_\gamma<\oo$, set
\[
D_f{(\gamma)}=\frac{b_\gamma-\gamma}{b_\gamma-a_\gamma}\chi_{N_{a_\gamma}}+\frac{\gamma-a_\gamma}{b_\gamma-a_\gamma}\chi_{N_{b_\gamma}}.
\]
Otherwise, set $D_f{(\gamma)}=\chi_{N_{a_\gamma}}$. Evidently, $D_f(\gamma)\in L^\downarrow(\mathcal A)$ and $0\le D_f(\gamma)\le1$. If $0\le g\in L^1_\mu+L^\oo_\mu\!\!\downarrow$ and $\int_Mg\,d\mu=\int_Mf\,d\mu$ for all $M\in \mathcal M$, then $f$ and $g$ give rise to the same $\Theta$, the same $N_x$ and the same $a_\gamma$ and $b_\gamma$. Therefore $D_f=D_g$. 

To prove the last statement of the lemma, we let $\gamma=\|f_1\|_{L^1_\mu}$ and, for convenience, write $a=a_\gamma$ and $b=b_\gamma$. First we show that $\|D_f{(\gamma)}f\|_{L^1_\mu}\le\|f_1\|_{L^1_\mu}$: If $D_f(\gamma)=\chi_{N_{a}}$, then
\[
\|D_f{(\gamma)}f\|_{L^1_\mu}=\int_{N_{a}}f\,d\mu=a\le\gamma=\|f_1\|_{L^1_\mu}.
\]
Otherwise,
\[
\|D_f{(\gamma)}f\|_{L^1_\mu}=\frac{b-\gamma}{b-a}\int_{N_{a}}f\,d\mu+\frac{\gamma-a}{b-a}\int_{N_{b}}f\,d\mu=\gamma=\|f_1\|_{L^1_\mu}.
\]
On the way to proving that $\|(1-D_f(\gamma))f\|_{L^\oo_\mu\!\downarrow}\le\|f_\oo\|_{L^\oo_\mu\!\downarrow}$ we show that, for all $M\in\mathcal M$, 
\begin{equation}\label{gamma}
\int_M f_1 \,d\mu\le\int_M D_f{(\gamma)}f \,d\mu.
\end{equation}
Fix $M\in\mathcal M$. The definition of $\Theta(\mathcal M)$ ensures that $\int_M f \,d\mu\le a$ or $\int_M f \,d\mu\ge b$. 

Case 1. $a<\gamma<b<\oo$: If $\int_M f \,d\mu\le a$, then
\[
\int_M f \,d\mu\le\int_{N_{a}} f \,d\mu\le\int_{N_{b}} f \,d\mu\ \ \text{so}\ \ 
\int_M f \,d\mu=\int_{M\cap N_{a}} f \,d\mu=\int_{M\cap N_{b}} f \,d\mu,
\]
because $\mathcal M$ is totally ordered. Thus,
\[
\int_M f_1 \,d\mu\le\int_M f \,d\mu=\frac{b-\gamma}{b-a}\int_{M\cap N_{a}} f \,d\mu+\frac{\gamma-a}{b-a}\int_{M\cap N_{b}} f \,d\mu=\int_MD_f{(\gamma)}f \,d\mu.
\]
If $\int_M f \,d\mu\ge b$, then
\[
\int_{N_{a}} f \,d\mu\le\int_{N_{b}} f \,d\mu\le\int_M f \,d\mu.
\]
Since $\mathcal M$ is totally ordered,
\[
a=\int_{N_{a}} f \,d\mu=\int_{M\cap N_{a}} f \,d\mu\quad\text{and}\quad
b=\int_{N_{b}} f \,d\mu=\int_{M\cap N_{b}} f \,d\mu.
\]
Therefore,
\[
\int_M f_1 \,d\mu\le\gamma=\frac{b-\gamma}{b-a}\int_{M\cap N_{a}} f \,d\mu+\frac{\gamma-a}{b-a}\int_{M\cap N_{b}} f \,d\mu=\int_M D_f{(\gamma)}f \,d\mu.
\]

Case 2. $a<\gamma<b<\oo$ fails: If $\int_M f \,d\mu\le a$, then
\[
\int_M f \,d\mu\le\int_{N_{a}} f \,d\mu\quad \text{so}\quad
\int_M f \,d\mu=\int_{M\cap N_{a}} f \,d\mu,
\]
because $\mathcal M$ is totally ordered. Thus,
\[
\int_M f_1 \,d\mu\le\int_M f \,d\mu=\int_{M\cap N_{a}} f \,d\mu=\int_M D_f{(\gamma)}f \,d\mu.
\]
If $\int_M f \,d\mu\ge b$, then $a=\gamma=b$ so, because $\mathcal M$ is totally ordered,
\[
\int_M f_1 \,d\mu\le\gamma=\int_{N_{a}} f \,d\mu=\int_{M\cap N_{a}} f \,d\mu=\int_M D_f{(\gamma)} f\,d\mu.
\]
This completes the proof of \eqref{gamma}. 

Using \eqref{gamma}, we get
\[
\int_M (1-D_f{(\gamma)})f \,d\mu=\int_M f \,d\mu-\int_M D_f{(\gamma)}f \,d\mu\le\int_M f \,d\mu-\int_M f_1 \,d\mu\le\int_M f_\oo \,d\mu
\]
for each $M\in\mathcal M$ and it follows that 
\[
\| (1-D_f{(\gamma)})f\|_{L^\oo_\mu\!\downarrow}\le\| f_\oo\|_{L^\oo_\mu\!\downarrow}.
\]
\end{proof}

The properties of $D_f(\gamma)$ are the key to proving a formula for the $K$-functional of $(L^1_\mu,L^\oo_\mu\!\!\downarrow)$.
\begin{theorem}\label{Ko} Let $0\le f\in L^1_\mu+L^\oo_\mu\!\!\downarrow$ and $t>0$. Then
\[
K(f,t;L^1_\mu,L^\oo_\mu\!\!\downarrow)=K(QRf,t;L^1_\mu,L^\oo_\mu\!\!\downarrow)=K(Rf,t;L^1_\lambda,L^\oo_\lambda\!\!\downarrow)=\int_0^t(f^o)^*.
\]
\end{theorem}
\begin{proof} It follows from \eqref{Kpos} and Lemma \ref{decomp} that
\begin{equation}\label{K}
K(f,t;L^1_\mu,L^\oo_\mu\!\!\downarrow)=\inf_{\gamma\ge0}\| D_f{(\gamma)}\| _{L^1_\mu}+t\| (1-D_f{(\gamma)})f\|_{L^\oo_\mu\!\downarrow}.
\end{equation}
By \eqref{QRflip}, $\int_MQRf\,d\mu=\int_Mf\,d\mu$ for all $M\in \mathcal M$. So by Lemma \ref{decomp}, $D_{QRf}=D_f$ and $D_f\in L^+(\mathcal A)$. It follows that $(1-D_f)\chi_M\in L^+(\mathcal A)$ so, by Propositions \ref{R}\ref{Rc} and \ref{Q}\ref{Qc}, and Theorem \ref{QR}\ref{QRb},
\[
\int_UD_f(\gamma)QRf\,d\mu=\int_UfQR(D_f(\gamma))\,d\mu=\int_UfD_f(\gamma)\,d\mu
\]
and, for each $M\in\mathcal M$,
\[
\int_M(1-D_f(\gamma))QRf\,d\mu
=\int_UfQR((1-D_f(\gamma))\chi_M)\,d\mu=\int_Mf(1-D_f(\gamma))\,d\mu.
\]
We conclude that
\begin{align*}
\|D_f(\gamma)QRf\|_{L^1_\mu}&=\|D_f(\gamma)f\|_{L^1_\mu}\quad\text{and}\\
\|(1-D_f(\gamma))QRf\|_{L^\oo_\mu\!\downarrow}&=\|(1-D_f(\gamma))f\|_{L^\oo_\mu\!\downarrow}.
\end{align*}
Using these in the right-hand side of \eqref{K}, we get 
\[
K(f,t;L^1_\mu,L^\oo_\mu\!\!\downarrow)=K(QRf,t;L^1_\mu,L^\oo_\mu\!\!\downarrow).
\]

Using Lemma \ref{AM}, two applications of \eqref{T}, yield
\[
K(QRf,t;L^1_\mu,L^\oo_\mu\!\!\downarrow)\le K(Rf,t;L^1_\lambda,L^\oo_\lambda\!\!\downarrow)\le K(f,t;L^1_\mu,L^\oo_\mu\!\!\downarrow)
\]
so we have equality throughout.

Since $f\in L^+(\Sigma)$, $Rf\in L^+(\mathcal B)$ by Proposition \ref{R}\ref{Rb}. Let $(Rf)^o\in L^\downarrow(\mathcal B)$ be the level function of Section \ref{old}. From the proof of Lemma \ref{level}, $f^o=Q((Rf)^o)\in L^\downarrow(\mathcal A)$ is a level function of $f$. By Theorem \ref{QR}\ref{QRd}, $((Rf)^o)^*=(f^o)^*$, so from Section \ref{old}, we have
\[
K(Rf,t;L^1_\lambda,L^\oo_\lambda\!\!\downarrow)=\int_0^t((Rf)^o)^*=\int_0^t(f^o)^*.
\]
This completes the proof.
\end{proof}

Calculation of the $K$-functional for $(\widetilde{L^1_\mu},L^\oo_\mu)$ relies on properties of the rearrangement. 
\begin{theorem}\label{Kt} If $0\le g\in \widetilde{L^1_\mu}+L^\oo_\mu$ and $t>0$, then
\[
K(QRg,t;\widetilde{L^1_\mu},L^\oo_\mu)\le K(Rg,t;\widetilde{L^1_\lambda},L^\oo_\lambda)\le K(g,t;\widetilde{L^1_\mu},L^\oo_\mu)=\int_0^t(\widetilde g)^*,
\]
with equality throughout when $g\in L^+(\mathcal A)$.
\end{theorem}
\begin{proof} Let $0\le g\in\widetilde{L^1_\mu}+L^\oo_\mu$. Using Lemma \ref{AM}, two applications of \eqref{T}, yield the first two inequalities above. They hold with equality if $g\in L^+(\mathcal A)$, because Theorem \ref{QR}\ref{QRb} applies and we have $QRg=g$. It remains to prove that $K(g,t;\widetilde{L^1_\mu},L^\oo_\mu)=\int_0^t(\widetilde g)^*$.

Suppose $0\le g_1\in\widetilde{L^1_\mu}$, $0\le g_\oo\in L^\oo_\mu$, and $g=g_1+g_\oo$. Sublinearity of the least core decreasing majorant, see Theorem \ref{props}, shows that $\widetilde g\le \widetilde{g_1}+\widetilde{g_\oo}$. By \cite{BS}*{Proposition 2.1.7}, for every $\varepsilon\in(0,1)$,
\[
\int_0^t(\widetilde g)^*(s)\,ds\le\int_0^t(\widetilde{g_1}+\widetilde{g_\oo})^*(s)
\,ds\le\int_0^t(\widetilde{g_1})^*((1-\varepsilon)s)\,ds+\int_0^t(\widetilde{g_\oo})^*(\varepsilon s)\,ds.
\]
Now
\[
\int_0^t(\widetilde{g_1})^*((1-\varepsilon)s)\,ds=\frac1{1-\varepsilon}\int_0^{(1-\varepsilon)t}(\widetilde{g_1})^*(s)\,ds\le\frac1{1-\varepsilon}\int_U\widetilde{g_1}\,d\mu=\frac1{1-\varepsilon}\|g_1\|_{\widetilde{L^1_\mu}}
\]
and
\[
\int_0^t(\widetilde{g_\oo})^*(\varepsilon s)\,ds\le t\|g_\oo\|_{L^\oo_\mu}.
\]
Letting $\varepsilon\to0$ we get
\[
\int_0^t(\widetilde g)^*(s)\,ds\le\|g_1\|_{\widetilde{L^1_\mu}}+t\|g_\oo\|_{L^\oo_\mu}
\]
and, taking the infimum over all such decompositions of $g$, we conclude that
\[
\int_0^t(\widetilde g)^*\le K(g,t;\widetilde{L^1_\mu},L^\oo_\mu).
\]

For the reverse inequality, fix $t>0$ and set $y=(\widetilde g)^*(t)$. Let $g_1=\max(0,g-y)$. Since $g\le\widetilde g\in L^\downarrow(\mathcal A)$,  $g_1\le\max(0,\widetilde g-y)\in L^\downarrow(\mathcal A)$. If $g_1\le h\in L^\downarrow(\mathcal A)$, then $g\le h+y\in L^\downarrow(\mathcal A)$ so $\widetilde g\le h+y$. This shows that $\widetilde {g_1}=\max(0,\widetilde g-y)$. Since $(\widetilde g)^*\ge y$ on $[0,t]$ and $(\widetilde g)^*\le y$ on $[t,\oo)$, we can apply \cite{BS}*{Exercise 2.1}, to get
\[
\|g_1\|_{\widetilde{L^1_\mu}}=\int_U\max(0,\widetilde g-y)\,d\mu
=\int_0^\oo\max(0,(\widetilde g)^*-y)=\int_0^t(\widetilde g)^*- y.
\] 

Evidently, $g-g_1=\min(g,y)\le y$ so $\|g-g_1\|_{L^\oo_\mu}\le y$ and we conclude that
\[
K(g,t;\widetilde{L^1_\mu},L^\oo_\mu)\le\int_0^t((\widetilde g)^*- y)+ty=\int_0^t(\widetilde g)^*.
\]

\end{proof}

The next two theorems show that both $(L^1_\mu,L^\oo_\mu\!\!\downarrow)$ and $(\widetilde{L^1_\mu},L^\oo_\mu)$ are exact Calder\'on couples.
\begin{theorem}\label{CClev} $(L^1_\mu,L^\oo_\mu\!\!\downarrow)$ is an exact Calder\'on couple.
\end{theorem}
\begin{proof} Suppose $f,g\in L^1_\mu+L^\oo_\mu\!\!\downarrow$ satisfy $K(f,t;L^1_\mu,L^\oo_\mu\!\!\downarrow)\le K(g,t;L^1_\mu,L^\oo_\mu\!\!\downarrow)$ for all $t>0$. To show that $(L^1_\mu,L^\oo_\mu\!\!\downarrow)$ is an exact Calder\'on couple we need to find an admissible contraction from $(L^1_\mu,L^\oo_\mu\!\!\downarrow)$ to itself that sends $g$ to $f$. 

Write $f=(\sgn f)|f|$ and $|g|=(\sgn g)g$ and observe that the maps $h\mapsto (\sgn g)h$ and $h\mapsto (\sgn f)h$ are both admissible contractions from $(L^1_\mu,L^\oo_\mu\!\!\downarrow)$ to itself. The first takes $g\to|g|$ and the second takes $|f|\to f$. It remains to find an admissible contraction that sends $|g|$ to $|f|$. This reduces the problem to the case $f\ge0$ and $g\ge0$, which we assume in what follows.

We will exhibit three admissible contractions, mapping $g$ to $Rg$, then $Rg$ to $Rf$, and finally $Rf$ to $f$. By Lemma \ref{AM}, $R$ is an admissible contraction from  $(L^1_\mu,L^\oo_\mu\!\!\downarrow)$ to $(L^1_\lambda,L^\oo_\lambda\!\!\downarrow)$ so the step that maps $g$ to $Rg$ just uses $R$. 

By Theorem \ref{Ko},
\[
K(Rf,t;L^1_\lambda,L^\oo_\lambda\!\!\downarrow)\le K(Rg,t;L^1_\lambda,L^\oo_\lambda\!\!\downarrow).
\]
But $(L^1_\lambda,L^\oo_\lambda\!\!\downarrow)$ is a Calder\'on couple, see Section \ref{old}, so there is an admissible contraction $H$ from $(L^1_\lambda,L^\oo_\lambda\!\!\downarrow)$ to itself such that $HRg=Rf$, and the second step is complete.

By Proposition \ref{R}\ref{Rb}, $Rf\ge0$. For $x\ge0$, let $w(x)=1/Rf(x)$ when $Rf(x)\ne0$ and $w(x)=0$ when $Rf(x)=0$. Set $W\psi=fQ(w\psi)$ for all nonnegative $\psi\in L^1_\lambda+L^\oo_\lambda\!\!\downarrow$. Evidently, $W$ is additive and positive homogeneous. For each $M\in\mathcal M$,
\[
\int_MW\psi\,d\mu=\int_{[0,\mu(M)]}(Rf)w\psi\,d\lambda\le\int_{[0,\mu(M)]}\psi\,d\lambda.
\]
Since $\mathcal A$ is a full ordered core, the monotone convergence theorem shows 
\begin{equation}\label{W1}
\|W\psi\|_{L^1_\mu}\le\|\psi\|_{L^1_\lambda}.
\end{equation}
By Section \ref{old}, Theorem \ref{Loo}\ref{Loob}, and by Lemma \ref{lambda}, which shows $\lambda([0,\mu(M)])=\mu(M)$, we also have  
\begin{equation}\label{Woo}
\|W\psi\|_{L^\oo_\mu\!\downarrow}\le\|\psi\|_{L^\oo_\lambda\!\downarrow}.
\end{equation}
A standard argument shows that $W$ extends to be a linear operator from $L^1_\lambda+L^\oo_\lambda\!\!\downarrow$ to $L^1_\mu+L^\oo_\mu\!\!\downarrow$, and the norm inequalities \eqref{W1} and \eqref{Woo} remain valid. The extended $W$ is an admissible contraction from $(L^1_\lambda,L^\oo_\lambda\!\!\downarrow)$ to $(L^1_\mu,L^\oo_\mu\!\!\downarrow)$.  To see that $WRf=f$ we observe that $wRf\le1$ so $WRf=fQ(wRf)\le f$ by Proposition \ref{Q}\ref{Qh}. Also for each $M\in\mathcal M$,
\[
\int_MWRf\,d\mu=\int_{[0,\mu(M)]}(Rf)wRf\,d\lambda=\int_{[0,\mu(M)]}Rf\,d\lambda=\int_Mf\,d\mu.
\]
Since $\mathcal A$ is a full ordered core, it follows that $WRf=f$, $\mu$-a.e. Thus, the third step uses the extended map $W$.

Composing these maps, we see that $WHR$ is an admissible contraction from $(L^1_\mu,L^\oo_\mu\!\!\downarrow)$ to itself such that $WHRg=f$.
\end{proof}

\begin{theorem}\label{CClcdm} $(\widetilde{L^1_\mu},L^\oo_\mu)$ is an exact Calder\'on couple.
\end{theorem}
\begin{proof} Suppose $f,g\in L^1_\mu+L^\oo_\mu\!\!\downarrow$ satisfy $K(f,t;\widetilde{L^1_\mu},L^\oo_\mu)\le K(g,t;\widetilde{L^1_\mu},L^\oo_\mu)$ for all $t>0$. To show that $(\widetilde{L^1_\mu},L^\oo_\mu)$ is an exact Calder\'on couple we need to find an admissible contraction from $(\widetilde{L^1_\mu},L^\oo_\mu)$ to itself that sends $g$ to $f$. 

As in the proof of Theorem \ref{CClev}, we may assume that $f$ and $g$ are nonnegative.  

Since $\widetilde g\in L^+(\mathcal A)$, $\widetilde{\widetilde f\phantom{.}}\!=\widetilde f$, and $\widetilde{\widetilde g}=\widetilde g$, Theorem \ref{Kt} shows that
\[
K(R\widetilde f,t;\widetilde{L^1_\lambda},L^\oo_\lambda)=\int_0^t(\widetilde f)^*\le\int_0^t(\widetilde g)^*=K(R\widetilde g,t;\widetilde{L^1_\lambda},L^\oo_\lambda).
\]
From Section \ref{old}, $(\widetilde{L^1_\lambda},L^\oo_\lambda)$ is a Calder\'on couple so there exists an admissible contraction $H$ from $(\widetilde{L^1_\lambda},L^\oo_\lambda)$ to itself such that $HR\widetilde g=R\widetilde f$. Lemma \ref{AM} implies that $QHR$ is an admissible contraction from $(\widetilde{L^1_\mu},L^\oo_\mu)$ to itself and, since $\widetilde f\in L^\downarrow(\mathcal A)$, Theorem \ref{QR}\ref{QRb} shows that $QHR\widetilde g=QR\widetilde f=\widetilde f$. 

It remains to find admissible contractions $W_1$ and $W_2$, from $(\widetilde{L^1_\mu},L^\oo_\mu)$ to itself, such that $W_2\widetilde f=f$ and $W_1g=\widetilde g$, since then $W_2QHRW_1$ will be the desired map.

Let $\theta(s)=f(s)/\widetilde f(s)$ when $\widetilde f(s)\ne0$ and $\theta(s)=0$ otherwise. Then let $W_2h=\theta h$ and note that $W_2\widetilde f=f$. Since $|\theta|\le  1$, $W_2$ is an admissible contraction from $(\widetilde{L^1_\mu},L^\oo_\mu)$ to itself. 

The map $W_1$ is constructed in the same way as the map $W_1$ in the proof of \cite{MS2}*{Theorem 2.3}. The argument is short so we repeat it here. On the one-dimensional subspace $\mathbb Rg$ of $\widetilde{L^1_\mu}+L^\oo_\mu$ the map $W_1(\alpha g)=\alpha\widetilde g$ is trivially linear and satisfies $W_1h\le\widetilde h$ for $h\in\mathbb Rg$. Theorem \ref{props} shows that $h\mapsto\widetilde h$ is sublinear so we may apply the Hahn-Banach-Kantorovich theorem, \cite{AB}*{Theorem 1.25}, to extend $W_1$ to be linear on $\widetilde{L^1_\mu}+L^\oo_\mu$ so that $W_1h\le\widetilde h$ remains valid. At $-h$ it gives $-W_1h\le\widetilde{-h}=\widetilde h$ so $|W_1h|\le\widetilde h$. Thus $\|W_1h\|_{L^\oo_\mu}\le \|\widetilde h\|_{L^\oo_\mu}=\|h\|_{L^\oo_\mu}$ and $\|W_1h\|_{\widetilde{L^1_\mu}}\le\|\widetilde h\|_{\widetilde{L^1_\mu}}=\|h\|_{\widetilde{L^1_\mu}}$, so $W_1$ is an admissible contraction from $(\widetilde{L^1_\mu},L^\oo_\mu)$ to itself that maps $g$ to $\widetilde g$. This completes the proof.
\end{proof}

To get the best result from the fact that $(L^1_\mu,L^\oo_\mu\!\!\downarrow)$ is a Calder\'on couple, we need to find the divisibility constant for the couple. 
\begin{lemma} The divisibility constant of $(L^1_\mu,L^\oo_\mu\!\!\downarrow)$ is $1$.
\end{lemma}
\begin{proof}  Suppose $0\le f\in L^1_\mu+L^\oo_\mu\!\!\downarrow$ and $\omega_j$ are nonnegative, concave functions on $(0,\oo)$ that satisfy $\sum_{j=1}^\oo\omega_j(1)<\oo$ and $K(f,t;L^1_\mu,L^\oo_\mu\!\!\downarrow)\le\sum_{j=1}^\oo\omega_j(t)$, for all $t>0$. By Proposition \ref{R}\ref{Rb},  Lemma \ref{AM}, and Theorem \ref{Ko} we have $0\le Rf\in L^1_\lambda+L^\oo_\lambda\!\!\downarrow$ and
\[
K(Rf,t;L^1_\lambda,L^\oo_\lambda\!\!\downarrow)=K(f,t;L^1_\mu,L^\oo_\mu\!\!\downarrow)\le\sum_{j=1}^\oo\omega_j(t).
\]
By \cite{MS1}*{Corollaries 3.9 and 4.7}, $(L^1_\lambda,L^\oo_\lambda\!\!\downarrow)$ has divisibility constant $1$. So, there exist functions $\varphi_j\in L^1_\lambda+L^\oo_\lambda\!\!\downarrow$ such that $K(\varphi_j,t;L^1_\lambda,L^\oo_\lambda\!\!\downarrow)\le\omega_j(t)$, for all $j$ and $t$, and $\sum_{j=1}^\oo\varphi_j$ converges to $Rf$ in $L^1_\lambda+L^\oo_\lambda\!\!\downarrow$.
Because $Rf$ is nonnegative, we  may assume that $\varphi_j\ge0$ for all $j$, since otherwise we can replace them with $\psi_j$, defined by $\psi_1=\min(|\varphi_1|,Rf)$ and $\psi_{n+1}=\min(|\varphi_{n+1}|,Rf-(\psi_1+\dots+\psi_n))$ for $n=1,2,\dots$.
By \eqref{QRflip} and \eqref{overM},  
\[
\int_MQRf\,d\mu=\int_Mf\,d\mu<\oo
\]
for all $M\in\mathcal M$. Since $QRf$ is $\sigma(\mathcal A)$-measurable, $\{u\in U:QRf(u)=0\}$ is $\sigma(\mathcal A)$-measurable and hence 
\[
0=\int_{\{u\in U:QRf(u)=0\}}QRf\,d\mu=\int_{\{u\in U:QRf(u)=0\}}f\,d\mu,
\]
whence $\{u\in U:QRf(u)=0\}$ is $\mu$-almost contained in $\{u\in U: f(u)= 0\}$. 

Set $f_j=(fQ\varphi_j)/QRf$ when $QRf$ is nonzero and $f_j=0$ otherwise. Then 
\begin{equation}\label{ptw}
\sum_{j=1}^\oo f_j=(fQRf)/QRf=f
\end{equation}
$\mu$-a.e. In particular, $0\le f_j\le f$ $\mu$-a.e for all $j$. Now fix $j$. 
The definition of $f_j$ implies that $fQ\varphi_j=(QRf)f_j$ $\mu$-a.e. and we get $QRf$, $Q\varphi_j\in L^+(\mathcal A)$ from Proposition \ref{Q}\ref{Qb}, so we may use Proposition \ref{R}\ref{Rj} and Theorem \ref{QR}\ref{QRa} to get 
\[
(Rf)RQ\varphi_j=R(fQ\varphi_j)=R((QRf)f_j)=(RQRf)Rf_j=(Rf)Rf_j.
\]
Since $f-f_j\ge0$, $Rf-Rf_j\ge0$ so $Rf_j=0$ whenever $Rf=0$. Therefore, we may cancel $Rf$ above to get $RQ\varphi_j\ge Rf_j$. This inequality, Lemma \ref{AM} and \eqref{T}, give
\[
K(Rf_j,t;L^1_\lambda,L^\oo_\lambda\!\!\downarrow)
\le K(RQ\varphi_j,t;L^1_\lambda,L^\oo_\lambda\!\!\downarrow)\le K(\varphi_j,t;L^1_\lambda,L^\oo_\lambda\!\!\downarrow)\le\omega_j(t).
\]
By Theorem \ref{Ko}, this gives $K(f_j,t;L^1_\mu,L^\oo_\mu\!\!\downarrow)\le\omega_j(t)$.
This estimate, along with \eqref{ptw} and the completeness of  $L^1_\mu+L^\oo_\mu\!\!\downarrow$, shows that for each $n$,
\[
\Big\|f-\sum_{j=1}^{n-1} f_j\Big\|_{L^1_\mu+L^\oo_\mu\!\downarrow}
=\Big\|\sum_{j=n}^\oo f_j\Big\|_{L^1_\mu+L^\oo_\mu\!\downarrow}
\le\sum_{j=n}^\oo \|f_j\|_{L^1_\mu+L^\oo_\mu\!\downarrow}
\le\sum_{j=n}^\oo\omega_j(1),
\]
where we have used the identity 
$\|f_j\|_{L^1_\mu+L^\oo_\mu\!\downarrow}=K(f_j,1;L^1_\mu,L^\oo_\mu\!\!\downarrow)$.  The right-hand side above is the tail of a convergent series, so it goes to zero as $n\to\oo$. We conclude that  $\sum_{j=1}^\oo f_j$ converges to $f$ in $L^1_\mu+L^\oo_\mu\!\!\downarrow$. 

To drop the nonnegativity assumption on $f$, construct the $f_j$ for $|f|$. It is easy to verify that the functions $\sgn(f)f_j$ give the required decomposition of $f$.
\end{proof}

We also need the divisibility constant for $(\widetilde{L^1_\mu},L^\oo_\mu)$. In \cite{MS2}*{Lemma 3.2 and Theorem 4.3} this is established for the usual order on the half line with Lebesgue measure and the tools are provided to extend it to $\sigma$-finite measures. We follow the method here.
\begin{lemma} The divisibility constant of $(\widetilde{L^1_\mu},L^\oo_\mu)$ is $1$.
\end{lemma}
\begin{proof} In this paper so far, a Borel function $\varphi$ on $[0,\oo)$ is viewed as $\lambda$-measurable and $\varphi^*$ denotes the rearrangement with respect to $\lambda$. In this proof we will also need to view such functions as Lebesgue measurable. To avoid confusion, we use $\varphi^\#$ to denote the rearrangement of $\varphi$ with respect to Lebesgue measure on $[0,\oo)$.

Recall the functions $a(x)$ and $b(x)$ introduced in Section \ref{tailor}.

If $x_1\ge\dots\ge x_K\ge x_{K+1}=0$ and $\alpha_1,\dots,\alpha_K$ are positive real numbers, then, for each $j$, $\varphi=\sum_{k=1}^K\alpha_k\chi_{[0,x_k]}$ takes the value $\sum_{k=1}^j\alpha_k$ on $(x_{j+1},x_j]$, a set of $\lambda$-measure $a(x_j)-a(x_j+1)$ according to Lemma \ref{lambda}. In the proof of that lemma, we find that $b(z)\le x$ if and only if $z\le a(x)$. Therefore, $\varphi\circ b$ takes that same value on $(a(x_{j+1}),a(x_j)]$, a set of Lebesgue measure $a(x_j)-a(x_j+1)$. These are the only nonzero values the two functions take, so $\varphi^*=(\varphi\circ b)^\#$. Taking limits of increasing sequences of such functions, and applying \cite{BS}*{Proposition 2.1.7}, we see that this formula remains valid when $\varphi$ is any nonnegative, nonincreasing function on $[0,\oo)$.

Suppose $0\le g\in \widetilde{L^1_\mu}+L^\oo_\mu$ and $\omega_j$ are nonnegative, concave functions on $[0,\oo)$ such that $\sum_{j=1}^\oo\omega_j(1)<\oo$ and $K(g,t;\widetilde{L^1_\mu},L^\oo_\mu)\le\sum_{j=1}^\oo\omega_j(t)$, for all $t$. Our object is to find $g_j\in \widetilde{L^1_\mu}+L^\oo_\mu$ such that $K(g_j,t;\widetilde{L^1_\mu},L^\oo_\mu))\le\omega_j(t)$, for all $j$ and $t$, and $\sum_{j=1}^\oo g_j$ converges to $g$ in $\widetilde{L^1_\mu}+L^\oo_\mu$.

Since $\widetilde g\in L^\downarrow(\mathcal A)$, $R\widetilde g\in L^\downarrow(\mathcal B)$. 
Since $b$ is nondecreasing, $(R\widetilde g)\circ b$ is nonnegative and nonincreasing. For each $t\ge0$, we apply the above rearrangement formula, Theorem \ref{QR}\ref{QRe} and Theorem \ref{Kt} to get
\[
\int_0^t((R\widetilde g)\circ b)^\#=\int_0^t(R\widetilde g)^*=\int_0^t(\widetilde g)^*=K(g,t;\widetilde{L^1_\mu},L^\oo_\mu)\le\sum_{j=1}^\oo\omega_j(t).
\]
Since the couple $(L^1,L^\oo)$ has divisibility constant $1$, there exist nonnegative Lebesgue measurable functions $\kappa_j\in L^1+L^\oo$ such that $\int_0^t\kappa_j^\#\le\omega_j(t)$, for all $j$ and $t$, and $\sum_{j=1}^\oo\kappa_j$ converges to $(R\widetilde g)\circ b$ in $L^1+L^\oo$. Using \cite{BS86}*{Theorem 5b} in the proof of divisibility shows that we may take $\kappa_j$ to be nonincreasing for each $j$.

By Lemma \ref{lambda}, $x\le b(x)$  and $b(b(x))=b(x)$ for all $x\ge0$. Thus
\[
0\le\sum_{j=1}^\oo \kappa_j(x)-\kappa_j(b(x))=(R\widetilde g)(b(x))-(R\widetilde g)(b(b(x)))=0
\]
and so $\kappa_j=\kappa_j\circ b$ for all $j$.

Set $g_j=(g/\widetilde g)Q\kappa_j$. Since $\kappa_j\in\ L^\downarrow(\mathcal B)$, Theorem \ref{QR}\ref{QRd} shows that $Q\kappa_j\in L^\downarrow(\mathcal A)$ and $(Q\kappa_j)^*=\kappa_j^*$. But $g\le\widetilde g$ so $g_j\le Q\kappa_j$ and we get $\widetilde{g_j}\le Q\kappa_j$ and $\big(\widetilde{g_j}\big)^*\le \kappa_j^*$. Thus, for all $j$ and $t$, Theorem \ref{Kt} shows
\[
K(g_j,t;\widetilde{L^1_\mu},L^\oo_\mu)=\int_0^t(\widetilde{g_j})^*\le\int_0^t\kappa_j^*=\int_0^t(\kappa_j\circ b)^*
=\int_0^t\kappa_j^\#\le\omega_j(t).
\]

By Lemma \ref{lambda}, $x=b(x)$ $\lambda$-a.e so  $(R\widetilde g)\circ b=R\widetilde g$ $\lambda$-a.e, Theorem \ref{QR}\ref{QRb} gives $Q((R\widetilde g)\circ b)=QR\widetilde g=\widetilde g$ $\mu$-a.e. Thus, by the Monotone Convergence Theorem,
\[
\sum_{j=1}^\oo g_j=(g/\widetilde g)Q\Big(\sum_{j=1}^\oo\kappa_j\Big)
=(g/\widetilde g)Q((R\widetilde g)\circ b)=g
\]
$\mu$-a.e. This, and the completeness of  $\widetilde{L^1_\mu}+L^\oo_\mu$, shows that for each $n$,
\[
\Big\|g-\sum_{j=1}^{n-1} g_j\Big\|_{\widetilde{L^1_\mu}+L^\oo_\mu}
=\Big\|\sum_{j=n}^\oo g_j\Big\|_{\widetilde{L^1_\mu}+L^\oo_\mu}
\le\sum_{j=n}^\oo \|g_j\|_{\widetilde{L^1_\mu}+L^\oo_\mu}
\le\sum_{j=n}^\oo\omega_j(1),
\]
where we have used the identity
$\|g_j\|_{\widetilde{L^1_\mu}+L^\oo_\mu}=K(g_j,1;\widetilde{L^1_\mu},L^\oo_\mu)$. The right-hand side above is the tail of a convergent series so it goes to zero as $n\to\oo$. We conclude that  $\sum_{j=1}^\oo g_j$ converges to $g$ in $\widetilde{L^1_\mu}+L^\oo_\mu$. 

To drop the nonnegativity assumption on $g$, construct the $g_j$ for $|g|$. It is easy to verify that the functions $\sgn(g)g_j$ give the required decomposition of $g$.
\end{proof}

We end by summarizing the interpolation properties of the couples $(L^1_\mu,L^\oo_\mu\!\!\downarrow)$ and $(\widetilde{L^1_\mu},L^\oo_\mu)$ as they relate to the universally rearrangement invariant spaces. See Section \ref{oldCC} for terminology.

\begin{theorem}\label{summ} Let $X$, $Y$, and $Z$ be Banach function spaces of $\mu$-measurable functions. Then
\begin{enumerate}[label=(\alph*)]
\item\label{summa} $X\in\Int_1(L^1_\mu, L^\oo_\mu)$ if and only if $X$ is u.r.i.;
\item\label{summb}  $Y\in \Int_1(L^1_\mu,L^\oo_\mu\!\!\downarrow)$ if and only if $Y=X^o$, with identical norms, for some u.r.i. space $X$;
\item\label{summc}  $Z\in \Int_1(\widetilde{L^1_\mu},L^\oo_\mu)$ if and only if $Z=\widetilde X$, with identical norms, for some u.r.i. space $X$.
\item\label{summd} If $X$ and $X'$ contain all characteristic functions of sets of finite measure, then $(X\!\!\downarrow)'=\widetilde{X'}$ with identical norms, $X\!\!\downarrow=(\widetilde{X'})'$ with identical norms, $\widetilde X\subseteq X\subseteq X''\subseteq X\!\!\downarrow$ and $(X'')^o\subseteq X\!\!\downarrow$;
\item\label{summe} If $X$ is u.r.i., then $X\!\!\downarrow=(X'')^o$ with identical norms, and $X\subseteq X^o$;
\item\label{summf} If $X$ is u.r.i.\ and has the Fatou property, then $X\!\!\downarrow=X^o$ with identical norms.
\end{enumerate}
\end{theorem}
\begin{proof} The following diagram may help to clarify the relationships described above. The proof continues below.

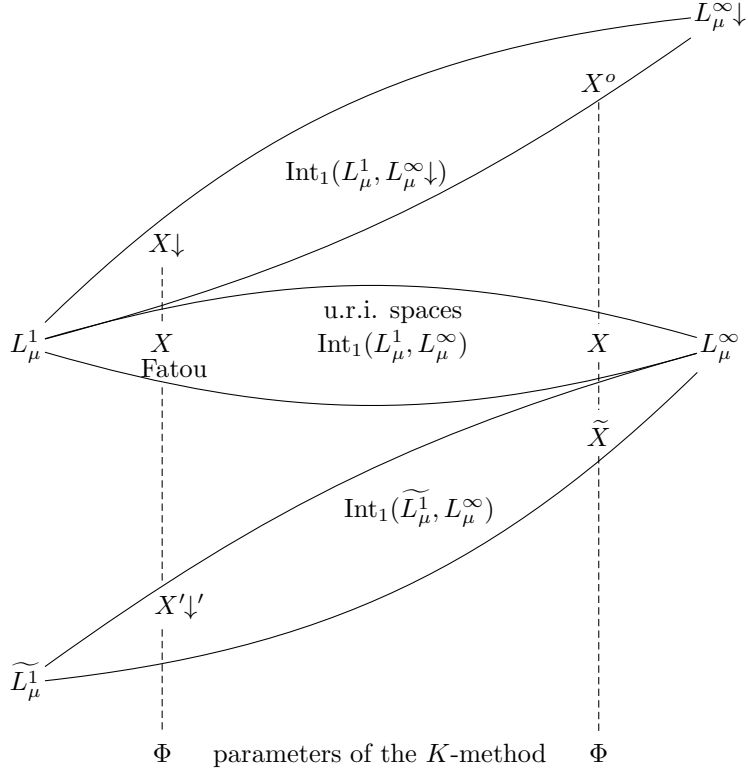
\begin{tikzcd}[row sep=small]
&&&&&
\\
&&&&&\!\!L^\oo_\mu\!\!\downarrow\\
&&&&X^o\arrow[dddd,dashed,no head]&\\
&&&&&\\
&&&\hbox to 0pt{\hss$\Int_1(L^1_\mu,L^\oo_\mu\!\!\downarrow)\qquad$\hss}&&\\
&&\ X\!\!\downarrow\arrow[d,dashed,no head]&&&\\
&L^1_\mu\!\!\arrow[uuuuurrrr,bend left=19,no head]\arrow[uuuuurrrr,bend right=10,no head]\arrow[rrrr,bend left=16,no head]\arrow[rrrr,bend right=16,no head]
&X\lower2.2ex\hbox to 0 pt{\hss Fatou\hss}\arrow[dddd,dashed,no head]
& \raise3ex\hbox to 0pt{\hss u.r.i. spaces\hss}\hbox to 0 pt{\hss$\Int_1(L^1_\mu,L^\oo_\mu)$\hss}
&X\arrow[d,dashed,no head]
&\!\!L^\oo_\mu\\
&&&&\widetilde X\arrow[ddddd,dashed,no head]&\\
&&&\qquad\Int_1(\widetilde{L^1_\mu},L^\oo_\mu)
&&\\&&&&\\
&&\ \quad X'\!\!\downarrow'\arrow[dd,dashed,no head]&&&\\
&\widetilde{L^1_\mu}\!\!\arrow[uuuuurrrr,bend left=10,no head]\arrow[uuuuurrrr,bend right=19,no head]&&&&\\
&&\Phi&\hbox to 0pt{\hss parameters of the $K$-method\ \ \ \hss}&\Phi&
\end{tikzcd}
\vskip 3ex
Consider the collection of all triples 
\[
\big((L^1_\mu,L^\oo_\mu)_\Phi, (L^1_\mu,L^\oo_\mu\!\!\downarrow)_\Phi,(\widetilde{L^1_\mu},L^\oo_\mu)_\Phi\big) 
\]
as $\Phi$ runs through all parameters of the $K$-method. The first entry of each triple is in $\Int_1(L^1_\mu, L^\oo_\mu)$, the second entry is in $\Int_1(L^1_\mu,L^\oo_\mu\!\!\downarrow)$ and the third entry is in $\Int_1(\widetilde{L^1_\mu},L^\oo_\mu)$. Part \ref{summa}, above, was proved in \cite{Ca}. It shows that the collection of first entries in these triples is exactly the set of u.r.i. spaces over $U$. Theorems \ref{Ko} and \ref{Kt} show that
\[
K(f,t;L^1_\mu,L^\oo_\mu\!\!\downarrow)=K(f^o,t;L^1_\mu,L^\oo_\mu)\quad\text{and}\quad
K(g,t;\widetilde{L^1_\mu},L^\oo_\mu)=K(\widetilde g,t;L^1_\mu,L^\oo_\mu).
\]
Therefore, every triple is of the form $(X,X^o,\widetilde X)$ for some u.r.i. space $X$. Conversely, every $(X,X^o,\widetilde X)$, where $X$ is u.r.i., is in the collection. Since both $(L^1_\mu,L^\oo_\mu\!\!\downarrow)$ and $(\widetilde{L^1_\mu},L^\oo_\mu)$ are exact Calder\'on couples with divisibility constant $1$, \cite{BK}*{Theorems 3.3.1 and 4.4.5 and Remark 4.4.4} show that every $Y\in \Int_1(L^1_\mu,L^\oo_\mu\!\!\downarrow)$ is the second entry in some triple, and every  $Z\in \Int_1(\widetilde{L^1_\mu},L^\oo_\mu)$ is the third entry in some triple. This proves \ref{summb} and \ref{summc}.

The results of \ref{summd} follow from Theorem \ref{props} and two of the remaining three statements appear in Theorem \ref{downlev}. We show that if $X$ is u.r.i., then $X\subseteq X^o$: Since $\|f\|_{L^\oo_\mu\!\downarrow}\le\|f\|_{L^\oo_\mu}$ for all $f$, $K(f,t;L^1_\mu,L^\oo_\mu\!\!\downarrow)\le K(f,t;L^1_\mu,L^\oo_\mu)$ for all $f$ and $t$. That is, $\int_0^t (f^o)^*\le\int_0^t f^*$ for all $f$ and $t$. Since $X$ is u.r.i.\ we get $\|f\|_{X^o}\le\|f\|_X$ for all $f$. This shows $X\subseteq X^o$.
\end{proof}

\end{document}